\documentclass[12pt]{amsart}
\usepackage{amssymb}
\usepackage{eucal}
\usepackage{bm}
\usepackage[all,cmtip]{xy}
\usepackage{color}
\usepackage[bookmarks=false]{hyperref}
\usepackage{multirow,bigdelim}
\usepackage[small,nohug,heads=vee]{diagrams}
\usepackage{url}
\usepackage{xtab}
\allowdisplaybreaks

\newtheorem{theorem}[equation]{Theorem}
\newtheorem{lemma}[equation]{Lemma}
\newtheorem{proposition}[equation]{Proposition}
\newtheorem{corollary}[equation]{Corollary}

\theoremstyle{definition}
\newtheorem{definition}[equation]{Definition}
\newtheorem{tablesymb}[equation]{Table of Symbols}
\newtheorem{example}[equation]{Example}

\theoremstyle{remark}
\newtheorem{remark}[equation]{Remark}

\numberwithin{equation}{subsection}

\allowdisplaybreaks[1]

\newcommand{\FF}{\mathbb{F}}
\newcommand{\ZZ}{\mathbb{Z}}
\newcommand{\QQ}{\mathbb{Q}}
\newcommand{\RR}{\mathbb{R}}
\newcommand{\LL}{\mathbb{L}}
\newcommand{\TT}{\mathbb{T}}
\newcommand{\GG}{\mathbb{G}}

\newcommand{\CC}{\mathbb{C}}

\newcommand{\KK}{\mathbb{K}}
\newcommand{\NN}{\mathbb{N}}

\newcommand{\ba}{\mathbf{a}}

\newcommand{\bff}{\mathbf{f}}
\newcommand{\bg}{\mathbf{g}}
\newcommand{\bh}{\mathbf{h}}
\newcommand{\bk}{\mathbf{k}}
\newcommand{\bmm}{\mathbf{m}}
\newcommand{\bn}{\mathbf{n}}
\newcommand{\bA}{\mathbf{A}}

\newcommand{\bx}{\mathbf{x}}

\newcommand{\bu}{\mathbf{u}}
\newcommand{\bv}{\mathbf{v}}

\newcommand{\bC}{\mathbf{C}}

\newcommand{\bz}{\mathbf{z}}

\newcommand{\bj}{\mathbf{j}}
\newcommand{\bw}{\mathbf{w}}

\newcommand{\bY}{\mathbf{Y}}

\newcommand{\cM}{\mathcal{M}}
\newcommand{\cN}{\mathcal{N}}

\newcommand{\cF}{\mathcal{F}}

\DeclareMathAlphabet{\matheur}{U}{eur}{m}{n}

\newcommand{\fs}{\mathfrak{s}}

 \DeclareMathOperator{\Lie}{Lie}
\DeclareMathOperator{\Ker}{Ker} 
\DeclareMathOperator{\Mat}{Mat} 
\DeclareMathOperator{\End}{End} 
\DeclareMathOperator{\Spec}{Spec}

 \DeclareMathOperator{\wt}{wt}
\DeclareMathOperator{\Ext}{Ext}  
\DeclareMathOperator{\Li}{Li}

\DeclareMathOperator{\fLi}{\mathfrak{Li}}

\DeclareMathOperator{\dep}{dep}
\DeclareMathOperator{\Span}{Span}

\newcommand{\oK}{\overline{K}}

\newcommand{\tr}{\mathrm{tr}}

\newcommand{\Lis}{\Li^{\star}}

\newcommand{\fLis}{\mathfrak{Li}^{\star}}
\newcommand{\zetaAmot}{\zeta_{A}^{\text{\tiny{\rm{AT}}}}}

\newcommand{\power}[2]{{#1 [\![ #2 ]\!]}}
\newcommand{\laurent}[2]{{#1 (\!( #2 )\!)}}

\definecolor{ForestGreen}{rgb}{0.0, 0.5, 0.0}

\newcommand{\C}{\ensuremath \mathbb{C}}
\newcommand{\Z}{\ensuremath \mathbb{Z}}

\newcommand{\N}{\ensuremath \mathbb{N}}
\newcommand{\F}{\ensuremath \mathbb{F}}

\newcommand{\bc}{\mathbf{c}}
\newcommand{\bd}{\mathbf{d}}

\newcommand{\inv}{\ensuremath ^{-1}}

\newcommand{\pinv}{\ensuremath ^{'-1}}

\DeclareMathOperator{\Exp}{Exp}
\DeclareMathOperator{\Log}{Log}

\DeclareMathOperator{\ord}{ord}

\newcommand{\twist}{^{(1)}}
\newcommand{\invtwist}{^{(-1)}}
\newcommand{\twistinv}{^{(-1)}}
\newcommand{\twisti}{^{(i)}}
\newcommand{\twistk}[1]{^{(#1)}}

\def\XXint#1#2#3{{\setbox0=\hbox{$#1{#2#3}{\int}$ }
\vcenter{\hbox{$#2#3$ }}\kern-.6\wd0}}

\makeatletter
\@namedef{subjclassname@2020}{MSC2020}
\makeatother

\title[Taylor coefficients of Anderson-Thakur series]{Taylor coefficients of Anderson-Thakur series and explicit formulae}

\author{Chieh-Yu Chang}

\address{Department of Mathematics, National Tsing Hua University, Hsinchu City 30042, Taiwan
  R.O.C.}
\email{cychang@math.nthu.edu.tw}

\author{Nathan Green}
\address{Department of Mathematics, University of California at San Diego, San Diego, CA 92093, USA}
\email{n2green@ucsd.edu}

\author{Yoshinori Mishiba}
\address{Department of Mathematical Sciences, University of the Ryukyus, 1 Senbaru, Nishihara-cho, Okinawa 903-0213, Japan}
\email{mishiba@sci.u-ryukyu.ac.jp}

\thanks{The first author was partially supported by MOST Grant
  107-2628-M-007-002-MY4 . The third author was supported by JSPS KAKENHI Grant Number JP18K13398.  This project was partially supported by JSPS Bilateral Open Partnership Joint Research Projects. }

\begin{document}

\begin{abstract}
For each positive characteristic multiple zeta value (defined by Thakur~\cite{T04}), the first and third authors  in~\cite{CM17b} constructed a $t$-module together with an algebraic point such that a specified coordinate of the logarithmic vector  of the algebraic point is a rational multiple of that multiple zeta value. The objective of this paper is to use the Taylor coefficients of Anderson-Thakur series and $t$-motivic Carlitz multiple star polylogarithms  to give explicit formulae for all of the coordinates of this logarithmic vector.
\end{abstract}

\subjclass[2020]{Primary 11M38; Secondary 11G09}
\keywords{Multiple zeta valutes, Anderson-Thakur series, t-motivic Carlitz multiple star polylogarithms, t-modules, logarithms of t-modules, t-motives}

\date{\today}
\maketitle

\setcounter{tocdepth}{1}
\tableofcontents

\section{Introduction}

In~\cite{CM17b}, the first and third authors of the present paper verify a function field analogue of Furusho's conjecture, asserting that the $p$-adic multiple zeta values (abbreviated as MZV's) satisfy the same $\mathbb{Q}$-linear relations that the corresponding real-valued MZV's satisfy. The key technique in that paper is to show that any $\infty$-adic MZV can be realized as a coordinate of the logarithm of an explicitly constructed $t$-module at a specific algebraic point. The main purpose of this paper is to give explicit formulae for all the coordinates of such a logarithmic vector in terms of Taylor coefficients of Anderson-Thakur series and $t$-motivic Carlitz multiple star polylogarithms. 

\subsection{Positive characteristic MZV's}  Let $t$ and $\theta$ be independent variables. Let $A:=\FF_{q}[\theta]$ be the polynomial ring in $\theta$ over a finite field $\FF_{q}$ with quotient field $K := \F_q(\theta)$. Let $K_{\infty}$ be the completion of $K$ with respect to the normalized absolute value $|\cdot|_{\infty}$ associated to the infinite place $\infty$, and let $\CC_{\infty}$ be the $\infty$-adic completion of a fixed algebraic closure of $K_{\infty}$.  Let $\overline{K}$ be the algebraic closure of $K$ inside $\CC_{\infty}$. Thakur~\cite{T04} defined positive characteristic multiple zeta values associated to $A$ as follows.  For any $r$-tuple of positive integers $\fs=(s_{1},\ldots,s_{r})\in \NN^{r}$, define
\begin{equation}\label{MZVdef}
\zeta_{A}(\fs):=\sum \frac{1}{a_{1}^{s_1}\cdots a_{r}^{s_r}}\in K_{\infty},
\end{equation}
where the sum is over all $r$-tuples of monic polynomials $a_{1},\ldots,a_{r}$ in $A$ with the restriction $|a_{1}|_{\infty}> |a_{2}|_{\infty}>\cdots>|a_{r}|_{\infty}$. The weight and the depth of $\zeta_{A}(\fs)$ are defined by  $\wt(\fs):=\sum_{i=1}^{r} s_{i}$ and $\dep(\fs):=r$ respectively. When $r=1$, the special values above are called Carlitz zeta values, due to Carlitz~\cite{Ca35}.

The classical real-valued MZV's
\[ \zeta(\fs):=\sum_{n_{1}>\cdots>n_{r}\geq 1}\frac{1}{n_{1}^{s_{1}}\cdots n_{r}^{s_{r}}} \in \RR^{\times}\]
for $\fs=(s_{1},\ldots,s_{r})\in \NN^{r}$ with $s_{1}\geq 2$,  are generalizations of special values of the Riemann $\zeta$-function at positive integers.  They occur as periods of mixed Tate motives and have many interesting connections between different research areas (see~\cite{An04, BGF18, Zh16}). Note that $\zeta(\fs)$ is the specialization at $z=1$ of the one-variable multiple polylogarithm 
\[ \sum_{n_{1}>\cdots>n_{r}\geq 1}\frac{z^{n_{1}}}{n_{1}^{s_{1}}\cdots n_{r}^{s_{r}}}  .\] In the function field setting,  this ``logarithmic" interpretation is replaced by a more subtle formula expressing $\zeta_A(\mathfrak s)$ as a $K$-linear combination of certain Carlitz multiple star polylogarithms (abbreviated as CMSPL's) evaluated at well-chosen integral points (by \cite{AT90} for depth one and \cite{C14} for arbitrary depth). See~\cite[Thm.~5.2.5]{CM17b}.

The Carlitz logarithm, denoted by $\Log_{\bC}$, plays an analogous role for the Carlitz module (see \ref{[t]ndef}) as the classical logarithm does for the multiplicative group $\GG_{m}$.  It is given by the series $\Log_{\bC}(z)=\sum_{i=0}^{\infty} \frac{z^{q^{i}}}{L_{i}}$, where $L_{0}:=1$ and $L_{i}:=(\theta-\theta^{q})\cdots (\theta-\theta^{q^{i}})$ for $i\in \NN$. We then define their deformation polynomials as

\begin{equation}\label{E: LLi}
 \LL_{0}:=1\hbox{ and }\LL_{i}:=(t-\theta^{q})\cdots(t-\theta^{q^{i}}) \hbox{ for }i\in\NN.
\end{equation}
For any index $\fs=(s_{1},\ldots,s_{r})\in \NN^{r}$ and $\bu=(u_{1},\ldots,u_{r})\in\oK^{r}$, the $\fs$-th ($\infty$-adic) $t$-motivic Carlitz multiple star polylogarithm associated to $\bu$ is defined by the series

\[ \fLis_{\fs,\bu}(t) = \sum_{i_1\geq \dots \geq i_r \geq 0} \frac{u_1^{q^{i_1}}\dots u_r^{q^{i_r}}}{\LL_{i_1}^{s_1}\dots \LL_{i_r}^{s_r}}\in \power{\oK}{t}, \] 
which is the deformation series of the CMSPL $\Lis_{\fs}$ at $\bu$ when $\Lis_{\fs}(\bu)$ is defined (see Sec.~\ref{Sub: t-motivic CMSPL}), i.e., $\fLis_{\fs,\bu}|_{t=\theta}=\Lis_{\fs}(\bu)$. The Anderson-Thakur series associated to $\fs$  was constructed by Anderson-Thakur~\cite{AT09} using the so-called Anderson-Thakur polynomials~\cite{AT90}  and we  denote it by $\zetaAmot(\fs)\in \power{\oK}{t}$ (see Def. \ref{Def: motivicMZV} for a precise definition).  Note that by~\cite{AT09}, $\zetaAmot(\fs)$ is a deformation of $\zeta_{A}(\fs)$ in the sense that $\zetaAmot(\fs)|_{t=\theta}=\Gamma_{\fs}\zeta_{A}(\fs)$, where $\Gamma_{\fs}:=\Gamma_{s_{1}}\cdots \Gamma_{s_{r}}\in A$ is the product of Carlitz factorials defined in \eqref{E: Carlitz factorials}.

We show in Lemma~\ref{L:Deformation identity} that $\zetaAmot(\fs)$ can be expressed using the following identity
\begin{equation}\label{E: IntroMotivicMZV}
\zetaAmot(\fs)=\sum_{\ell=1}^{T_\fs} b_{\ell}(t)\cdot (-1)^{\dep(\fs_{\ell})-1}\fLis_{\fs_{\ell},\bu_{\ell}}  
\end{equation}
for some positive integer $T_\fs$, explicit polynomials $b_{\ell}\in \FF_{q}[t]$, explicit indexes  $\fs_{\ell}\in \NN^{\dep(\fs_{\ell})}$ with $\dep(\fs_{\ell}) = 1$ ($1 \leq \ell \leq s$), $2 \leq \dep(\fs_{\ell}) \leq \dep(\fs)$ ($s+1 \leq \ell \leq T_{\fs}$) for some $1 \leq s \leq T_{\fs}$ and $\wt(\fs_{\ell})=\wt(\fs)$ and explicit integral points $\bu_{\ell}\in A^{\dep(\fs_{\ell})}$.  Note that specializing the identity \eqref{E: IntroMotivicMZV} at $t=\theta$ recovers the ``logarithmic" interpretation mentioned above:
\[ \Gamma_{\fs}\zeta_{A}(\fs)=\sum_{\ell=1}^{T_\fs} b_{\ell}(\theta)\cdot (-1)^{\dep(\fs_{\ell})-1}\Lis_{\fs_{\ell}}(\bu_{\ell}) .\]

\subsection{Statement of the main result}
In the world of global fields in positive characteristic, Anderson's $t$-modules (see~Sec.~\ref{Sec: background of t-modules}) play the role that commutative algebraic groups do in the classical theory.  In the classical setting, a productive direction of study for many transcendence questions has been to relate the special values one wishes to study to coordinates of the generalized logarithm of a commutative algebraic group over ${\overline{\mathbb{Q}}}$ evaluated at algebraic points. This method often meets with success because of  W\"ustholz's analytic subgroup theorem~\cite{W89}, which asserts that the $\overline{\QQ}$-linear relations among the coordinates of the generalized logarithm evaluated at a $\overline{\QQ}$-valued point of a commutative algebraic group defined over $\overline{\QQ}$ arise from the defining equations of $\Lie H$ for some algebraic subgroup $H \subset G$ defined over $\overline{\QQ}$.  This technique has led to many important applications in classical transcendental number theory (see~\cite{BW07}). In our positive characteristic setting, Yu's sub-$t$-module theorem~\cite{Yu97} plays the analogue of W\"ustholz's analytic subgroup theorem, and we hope to eventually study transcendence properties of Taylor coefficients of Anderson-Thakur series using our logarithmic interpretation given here.  For length considerations, we delay this to a future paper.

For  any index $\fs=(s_{1},\ldots,s_{r})\in \NN^{r}$, we let $\bu=(u_{1},\ldots,u_{r})\in \oK^{r}$ with $|u_i|_{\infty} \leq q^{\frac{s_i q}{q-1}}$ for each $1 \leq i \leq r-1$ and $|u_r|_{\infty} < q^{\frac{s_r q}{q-1}}$.  Following~\cite{CM17a}, we explicitly construct a $t$-module $G_{\fs,\bu}$ and an algebraic point $\bv_{\fs,\bu}\in G_{\fs,\bu}(\oK)$ in (\ref{E:Explicit t-moduleCMPL}) and (\ref{E:v_s,u}) using $\fs$ and $\bu$, and note that the logarithm of the $t$-module $G_{\fs,\bu}$, denoted by $\Log_{G_{\fs,\bu}}$, converges $\infty$-adically at $\bv_{\fs,\bu}$. We then establish the following explicit formula for this logarithmic vector $\Log_{G_{\fs,\bu}}(\bv_{\fs,\bu})$ (stated as Theorem~\ref{T:Logcoords} via Proposition~\ref{P: Taylor coeff} using hyperderivatives).

\begin{theorem}\label{T: IntrodT1}
Let the notation and assumptions be given as above.  For each $1\leq i\leq r$, we let $d_{i}:=s_{i}+\cdots+s_{r}$ and set $d:=d_{1}+\cdots+d_{r}$. For each $1\leq i\leq r$, we consider the Taylor expansion of the following $t$-motivic CMSPL at $t=\theta$:
\[\fLis_{(s_{r},\ldots,s_{i}), (u_{r},\ldots,u_{i})}(t)=\sum_{j=0}^{\infty}\alpha_{i,j}(t-\theta)^{j}. \]Define
\[
\bY_{\fs,\bu}:=\begin{pmatrix}
Y_{1}\\
\vdots\\
Y_{r}
\end{pmatrix}\in \Mat_{d\times 1}(\CC_{\infty}),
\]
where for each $1\leq i\leq r$, $Y_{i}$ is given by
\[
Y_{i}  = 
\begin{pmatrix}
(-1)^{r-i}\alpha_{i,d_{i}-1}\\
(-1)^{r-i}\alpha_{i,d_{i}-2}\\
\vdots\\
(-1)^{r-i} \alpha_{i,0}
\end{pmatrix}
\in \Mat_{d_i \times 1}(\CC_{\infty}).\]
Then we have the following formula 
\begin{equation}\label{D:logvector}
\Log_{G_{\fs,\bu}}(\bv_{\fs,\bu}) = \bY_{\fs,\bu}\in \Lie G_{\fs,\bu}(\CC_{\infty}).
\end{equation}
\end{theorem}

For any $r$-tuple $\Lambda=(\lambda_{1},\ldots,\lambda_{r})$, denote the reversed order $r$-tuple by $\widetilde{\Lambda}:=(\lambda_{r},\ldots,\lambda_{1})$. We explicitly construct the $t$-module $G_{\fs}$ over $K$, which arises from a fiber coproduct of the Anderson dual $t$-motives associated to the $t$-modules $G_{\widetilde{\fs_{\ell}},\widetilde{\bu_{\ell}}}$, and we construct a special point $\bv_{\fs}\in G_{\fs}(K)$ (see~Sec.~\ref{Subsec: fiber coproduct} for the detailed definitions).  In~\cite{CM17b}, the authors gave a logarithmic interpretation of $\zeta_{A}(\fs)$, in the sense that we construct a vector $Z_{\fs}\in \Lie G_{\fs}(\CC_{\infty})$ for which $\Gamma_{\fs}\zeta_{A}(\fs)$ occurs as the $\wt(\fs)$-th coordinate of $Z_{\fs}$.  The primary purpose of this paper is to  give explicit formulae for all the coordinates of $Z_{\fs}$. 

For any index $\fs\in\NN^{r}$, let the $\FF_{q}[t]$-module structure on $G_{\fs}$ (resp. $G_{\tilde{\fs}_{\ell},\tilde{\bu}_{\ell}}$) be defined by $\rho$ (resp.~$\rho_{\ell}$), and denote by $\partial\rho$ (resp.~$\partial\rho_{\ell}$) the induced map on $\Lie G_{\fs}$ (resp.~$\Lie G_{\tilde{\fs}_{\ell},\tilde{\bu}_{\ell}}$). Precise details are given in Sec.~\ref{Sec: background of t-modules}. For any vector $\bz\in \Mat_{n\times 1}(\CC_{\infty})$ with $n\geq d_{1}$, we denote by $\bz_{-}$ the vector of the last $n-d_{1}$ coordinates of $\bz$ (see Definition~\ref{Def: z-}). The main result of this paper is stated as follows (Theorem~\ref{T:MainThm}).

\begin{theorem}\label{T: IntrodT2} For any index $\fs=(s_{1},\ldots,s_{r})\in \NN^{r}$, we keep the notation in~\eqref{E: IntroMotivicMZV}. Let $d_{1}:=s_{1}+\cdots+s_{r}$ and let $Z_{\fs}\in \Lie G_{\fs}(\CC_{\infty})$ be the vector described as above. For each $s+1\leq \ell \leq T_{\fs}$, we  set $\bY_{\ell}:=\bY_{\tilde{\fs}_{\ell},\tilde{\bu}_{\ell}}$ given in Theorem~\ref{T: IntrodT1}. Considering the Taylor expansion of $\zetaAmot(\fs)$ at $t=\theta$,
\[\zetaAmot(\fs)=\sum_{j=0}^{\infty} a_{\fs,j}(t-\theta)^{j}   .\]
Then $Z_{\fs}$ has the following explicit formula 
\[
Z_{\fs}=\begin{pmatrix}
a_{\fs,d_{1}-1}\\
\vdots\\
a_{\fs,1}\\
a_{\fs,0}\\
\left(\partial \rho_{s+1}(b_{s+1})(\bY_{s+1})\right)_{-}\\
\vdots\\
\left( \partial \rho_{T_{\fs}}(b_{T_{\fs}}) (\bY_{T_{\fs}})\right)_{-}
\end{pmatrix}.
\]
\end{theorem}

Note that  $\partial \rho_{\ell}(b_{\ell})\left( \bY_{\ell}\right)$ can be explicitly written down (see Corollary~\ref{C: partial bY}) for each $s+1\leq \ell \leq T_{\fs}$, and the constant term $a_{\fs,0}$ is equal to $\Gamma_{\fs} \zeta_{A}(\fs)$, and the other $j$-th Taylor coefficient $a_{\fs,j}$ is given in terms of the $j$-th hyperderivative (with respect to $t$) of $\zetaAmot(\fs)$ evaluated at $t=\theta$.

\subsection{Some remarks} Here we supply some remarks concerning the above.

\begin{enumerate}

\item In many number theory settings, Taylor coefficients of important functions have interesting arithmetic interpretations.  In the setting of $t$-modules, Anderson and Thakur gave the first period interpretation of certain Taylor coefficients in~\cite{AT90}. There, they showed that for any positive integer $s$, the kernel of $\Exp_{\bC^{\otimes s}}$ is a free $A$-module of rank one, and the first $s$ Taylor coefficients of the Anderson-Thakur function $\omega_{s}$ (see \cite[Sec.~2.5]{AT90}) give the coordinates of the generator of $\Ker \Exp_{\bC^{\otimes s}}$  (see also~\cite[Lem.~8.3]{Ma18}).  A similar period interpretation occurs for the more general Drinfeld $\bA$-modules in \cite[Thm. 6.7]{G17a}. The results of Theorems~\ref{T: IntrodT1} and~\ref{T: IntrodT2} give a logarithmic interpretation for Anderson-Thakur series and $t$-motivic CMSPL's in terms of Taylor coefficients. 

\item When $\dep(\fs)=1$, ie., $\fs=s\in \NN$, our $t$-module $G_{s}$ is equal to the $s$-th power $\bC^{\otimes s}$ of the Carlitz module $\bC$, and our vector $Z_s$ and special point $\bv_{s}$ are equal to the ones given in~\cite{AT90}  (where they used the notation $z_{s}$ and $Z_{s}$ respectively). In this special case of depth one, Anderson-Thakur showed that the last coordinate of $Z_{s}$ is given by $\Gamma_{s}\zeta_{A}(s)$, but did not figure out the other coordinates of $Z_{s}$.  In~\cite{Pp}, Papanikolas gives formulae for all other coordinates of $Z_{s}$, but his formulae are different from ours in Theorem \ref{T: IntrodT2}. 

\item In~\cite{AT09},  for each index $\fs\in \NN^{r}$  Anderson and Thakur constructed a $t$-motive $M_{\fs}$ together with a system of Frobenius difference equations $\Psi_{\fs}^{(-1)}=\Phi_{\fs}\Psi_{\fs}$  so that $\Gamma_{\fs}\zeta_{A}(\fs)/\tilde{\pi}^{\wt(\fs)}$ occurs as an entry of the period matrix $\Psi_{\fs}|_{t=\theta}$ of $M_{\fs}$, where $\Phi_{\fs}$ is a matrix of size $r+1$ with entries in $\overline{K}[t]$, where $\Psi_{\fs}$ is an invertible matrix of size $r+1$ with entries in $\power{\CC_{\infty}}{t}$, and $\tilde{\pi}$ is a fundamental period of the Carlitz module $\bC$. Here the terminology of $t$-motive is in the sense of~\cite{P08}. The value $\Gamma_{\fs}\zeta_{A}(\fs)/\tilde{\pi}^{\wt(\fs)}$ is the specialization at $t=\theta$ of a certain power series $\Omega^{\wt(\fs)}\cdot \zetaAmot(\fs) \in \power{\C_\infty}{t}$, which is an entry of $\Psi_{\fs}$  (see \cite[\S 2.5]{AT09} for details). Papanikolas~\cite{P08} showed that $Z_{\Psi_{\fs}}:= \Spec \overline{K}(t)[\Psi_{\fs},\det \Psi_{\fs}^{-1}]$ is a torsor for the algebraic group $\Gamma_{M_{\fs}}\times_{\FF_{q}(t)} \overline{K}(t)$, which arises from the base change of the fundamental group of the Tannakian category generated by $M_{\fs}$, the so-called $t$-motivic Galois group of $M_{\fs}$. Since $\zetaAmot(\fs)$ occurs in the affine coordinate ring of $Z_{\Psi_{\fs}}$, which we regard as a period torsor for $\Gamma_{M_{\fs}}$,  it can be viewed as a $t$-motivic period (cf.~\cite[Def.~4.1]{Br14}). These ideas of $t$-motives and $t$-deformations of the logarithm originally go back to Drinfeld and Anderson (see \cite{A86} and also Anderson-Thakur~\cite{AT90}, Papanikolas~\cite{P08}), and they were later extensively studied in~\cite{CY07, AT09, C14, M17}.

\item  As far as we know, there is no analogue of Theorem~\ref{T: IntrodT2} in the classical case as it is not clear whether one can relate real-valued MZV's to coordinates of generalized logarithms of commutative algebraic groups defined over $\overline{\mathbb{Q}}$ at algebraic points.
  
\item  In \cite{Ma18}, Maurischat realizes hyperderivatives of the coordinates of periods of $t$-modules as coordinates of the period matrix of a specifically constructed $t$-module called a prolongation.  He then combines these techniques with Papanikolas' theory~\cite{P08} to show that the coordinates of periods of certain $t$-modules are algebraically independent.  Because the individual coordinates in our main theorem \ref{T:MainThm} can be expressed in terms of hyperderivatives, it would be natural to try applying Maurischat's techniques to study transcendence questions about MZV's.  Specifically, we hope to be able to prove that the coordinates of the logarithmic vector of Theorem \ref{T:MainThm} are algebraically independent.    However, in order to do this, we would need to compute the $t$-motivic Galois group associated to the $t$-motive.  This would be quite complicated as the size of $\mathcal{M}$ is quite large.  We hope to study this transcendence question in a future project.
\end{enumerate}

\subsection{Organization of the paper}
To establish the explicit formulae for $Z_{\fs}$ in Theorem~\ref{T: IntrodT2}, we must first go back to the construction of $Z_{\fs}$ in \eqref{E: Def of Zs}. This boils down to establishing explicit formulae for the logarithmic vector \eqref{D:logvector}, whose proof is the core and technical part of this paper. 

In Section~\ref{Sec: Anderson's theory} we review the theory of $t$-modules and dual $t$-motives as well as go over some fundamental theories of Anderson, particularly Theorem~\ref{T: Anderson} which is a crucial tool in the proof of Theorem~\ref{T: IntrodT1}. We further review how to construct the associated $t$-module from a dual $t$-motive using his theory of $t$-frames. In this section, we also define the central object of our study, a dual $t$-motive $M'$ whose associated logarithm function gives the CMSPL's we're studying.

The main goal of Section~\ref{Sec: Hyperderivatives of CMSPL} is to express the logarithmic vector associated to $M'$ in terms of CMSPLs.  This is done in Theorem~~\ref{T: IntrodT1} (stated as Theorem~\ref{T:Logcoords}).  The main ingredient in the proof of Theorem~\ref{T:Logcoords} is Proposition~\ref{P: special case of general formulae}, which relates the logarithmic vector to invariants of the $t$-motive $M'$.  The reader should note that Proposition \ref{P: special case of general formulae} is a special case of a technical lemma, Lemma~\ref{L:technical lemma}, and for continuity of exposition, we delay the proof of Lemma~\ref{L:technical lemma} (and thus also Proposition \ref{P: special case of general formulae}) until Section \ref{Sec: general formulae}.  Because Section \ref{Sec: general formulae} is very technical, we suggest that the reader might wish to skim this section during a first reading.

In Section~\ref{Section of t-motivic MZV's}, we first review Anderson-Thakur series and then derive a deformation identity for them in Lemma~\ref{L:Deformation identity}.  In section~\ref{S:ExplicitFormulae} we make the final connection between our logarithmic vector and Anderson-Thakur series.  We begin by reviewing fiber coproducts of dual $t$-motives and the construction of $G_{\fs}$ and $Z_{\fs}$ mentioned in the introduction. With the results from Theorem~\ref{T:Logcoords} and Lemma~\ref{L:Deformation identity}, we use techniques developed in~\cite{CM17b} to derive an explicit formula for $Z_{\fs}$, stated in our Main Theorem~\ref{T:MainThm}. Finally, in Theorem~\ref{T: Monomial of MZV's} we further relate any monomial of MZV's to a coordinate of the logarithm of a certain $t$-module evaluated at a special point, and give a description of the other coordinates of this logarithmic vector explicitly in terms of Taylor coefficients of $t$-motivic CMSPL's.

\subsection*{Acknowledgements}
We are grateful to M. Papanikolas for sharing his manuscript~\cite{Pp} with us, and to J. Yu for his many helpful suggestions and comments. The first and third authors thank National Center for Theoretical Sciences and Kyushu University for their financial support and hospitality.   Finally, we thank the referees for  their helpful comments that improve the exposition of this paper.

\section{Anderson's theory revisited}\label{Sec: Anderson's theory}
In this section we briefly review the theory of Anderson's $t$-modules and their connection with dual $t$-motives.  The key result of this section is a formula connecting the exponential function of a $t$-module with invariants of the dual $t$-motive, given in Theorem \ref{T: Anderson}.

\subsection{Notation and Frobenius twistings}

\begin{tablesymb}\label{tableofsymbols} \textup{We use the following symbols throughout this paper}.\\
${}$\\
\begin{xtabular}{l r l l}

$\NN$ & $=$ & the set of positive integers.\\
$q$ & $=$ &  a power of a prime number $p$.\\
$\F_q$& $=$ & the finite field of $q$ elements.\\
$t,\theta$& $=$ & independent variables. \\
$A$ &$=$ & $\F_q[\theta]$, the polynomial ring in the variable $\theta$ over $\FF_{q}$.\\
$A_{+}$ & $=$ & the set of monic polynomials in $A$. \\
$K$ &$=$ & $\F_q(\theta)$, the quotient field of $A$.\\
$\ord_\infty$ &= &  the normalized valuation of $K$  at the infinite place for which $\ord_{\infty}(1/\theta)=1$.\\
$\lvert\cdot\rvert_{\infty}$ &$=$ & $q^{-\ord_{\infty}(\cdot)}$, an absolute value on $K$.\\
$K_\infty$ &$=$ &$\laurent{\F_q}{1/\theta}$,  the completion of $K$ at the infinite place.\\
$\C_\infty$ &$=$ &$\widehat{\overline{K}}_\infty$,  the completion of an algebraic closure of $K_\infty$.\\
$\oK$ & $=$ &  the algebraic closure of $K$ in $\CC_{\infty}$.\\
$L_{i}$ & $=$ & $(\theta-\theta^{q})\cdots (\theta-\theta^{q^{i}})$ for $i\in \NN$, and $L_{0}:=1$\\
 $\LL_{i}$ & $=$ & $(t-\theta^{q})\cdots (t-\theta^{q^{i}})$ for $i\in \NN$, and $\LL_{0}:=1$\\ 
$\widetilde{\Lambda}$ &$=$ &$(\lambda_{r},\ldots,\lambda_{1})$ for any $r$-tuple $\Lambda=(\lambda_{1},\ldots,\lambda_{r})$ of elements of a nonempty set.\\
$\wt(\fs)$&= & $\sum_{i=1}^{r}s_{i}$ for an index $\fs=(s_{1},\ldots,s_{r})\in \NN^{r}$.\\
$\dep(\fs)$&= & $r$ for an index $\fs=(s_{1},\ldots,s_{r})\in \NN^{r}$.\\
\end{xtabular}
\end{tablesymb}

 We denote by $\power{\CC_{\infty}}{t}$ the ring of formal power series in the variable $t$ with coefficients in $\CC_{\infty}$, and denote by $\laurent{\CC_{\infty}}{t}$ the quotient field of $\power{\CC_{\infty}}{t}$.  We define $\TT_{\theta}$ to be the subring of $\power{\CC_{\infty}}{t}$ consisting of power series convergent on $|t|_{\infty}\leq |\theta|_{\infty}$.
Then $\TT_{\theta}$ is complete with respect to the Gauss norm by putting \[ \lVert f \rVert_{\theta}:=\max_{i} \left\{ |b_{i} \theta^{i}|_{\infty}  \right\} \] for $f=\sum_{i \geq 0} b_{i}t^{i}\in \TT_{\theta}$ -- for more details see \eqref{Tatealgdef}. Finally for any integer $n$, we define the $n$-th fold Frobenius twist on $\laurent{\CC_{\infty}}{t}$:
\[ \begin{matrix}
 \laurent{\CC_{\infty}}{t}& \rightarrow& \laurent{\CC_{\infty}}{t}\\
 f:=\sum a_{i} t^{i}&\mapsto& f^{(n)}:=\sum a_{i}^{q^{n}}t^{i}.  
\end{matrix} \]
We then extend these Frobenius twists to matrices over $\laurent{\CC_{\infty}}{t}$, i.e., for any matrix $B=(b_{ij})\in \Mat_{\ell \times m}(\laurent{\CC_{\infty}}{t})$, we define 
\[B^{(n)}:=(b_{ij}^{(n)})  .\]

For any $A$-subalgebra $R$ of $\CC_{\infty}$, we define the ring of twisted polynomials 
\[ \Mat_{d}(R)[\tau]:=\left\{ \sum_{i=0}^{\infty} \alpha_{i}\tau^{i}| \alpha_{i}\in \Mat_{d}(R)\ \forall i \ \textnormal{and} \ \alpha_{i}=0\hbox{ for }i\gg 0  \right\}  \]
subject to the relation: for $\alpha,\beta\in \Mat_{d}(R)$,
\[ \alpha \tau^{i}\cdot \beta \tau^{j}:= \alpha \beta^{(i)} \tau^{i+j}. \]
We put $R[\tau]:=\Mat_{1}(R)[\tau]$ and indeed we have the natural identity $\Mat_{d}(R[\tau])=\Mat_{d}(R)[\tau]$. For any $\phi=\sum_{i=0}^{\infty}\alpha_{i}\tau^{i}\in \Mat_{d}(R)[\tau]$, we put
\[ \partial \phi:=\alpha_{0} .\]

\subsection{Background on Anderson $t$-modules}\label{Sec: background of t-modules} Fix a positive integer $d$ and an $A$-subalgebra $R$ of $\CC_{\infty}$ with quotient field $F$. Let $\rho$ be an $\FF_{q}$-linear ring homomorphism 
\[ \rho:\FF_{q}[t]\rightarrow \Mat_{d}(R)[\tau] \]
so that $\partial \rho(t)-\theta I_{d}$ is a nilpotent matrix. We have the natural identification
\[\Mat_{d}(R)[\tau]\cong \End_{\FF_{q}}\left(  {\GG_{a}^{d}}_{/R}\right) ,\]
where the latter is the ring of $\FF_{q}$-linear endomorphisms over $R$ of the algebraic group scheme $ {\GG_{a}^{d}}_{/R}$ and $\tau$ is identified with the Frobenius operator that acts on $  {\GG_{a}^{d}}_{/R}$ by raising each coordinate to the $q$-th power. Thus, the map $\rho$ gives rise to an $\FF_{q}[t]$-module structure on $  {\GG_{a}^{d}}_{/R}(R')$ for any $R$-algebra $R'$.  By a $d$-dimensional $t$-module  defined over $R$, we mean the pair $G=(  {\GG_{a}^{d}}_{/R}, \rho)$, which has underlying group scheme $  {\GG_{a}^{d}}_{/R}$, with an $\FF_{q}[t]$-module structure via $\rho$. Note that since $\rho$ is an $\FF_{q}$-linear ring homomorphism, $\rho$ is uniquely determined by $\rho(t)$. A sub-$t$-module of $G$ is a connected algebraic subgroup of $G$ that is invariant under the $\FF_{q}[t]$-action. 

A basic example is the $n$-th tensor power of the Carlitz module denoted by $\bC^{\otimes n}:=(  {\GG_{a}^{n}}_{/A}, [\cdot]_{n})$ for $n\in \NN$ introduced by Anderson and Thakur~\cite{AT90}, where 
$[\cdot]_{n}:\FF_{q}[t]\rightarrow \Mat_{n}(A)[\tau]$ is the $\FF_{q}$-linear ring homomorphism given by
\begin{equation}\label{[t]ndef}
 [t]_{n}=\theta I_{n}+N_{n}+E_{n}\tau
\end{equation}
with $N_n$ the matrix with $1$'s in the first super diagonal and $0$'s elsewhere and $E_n$ the matrix with a single $1$ in the lower left corner and $0$'s elsewhere.
When $n=1$, $\bC:=\bC^{\otimes 1}$ is the so-called Carlitz $\FF_{q}[t]$-module. 

Fix a $t$-module $G=(  {\GG_{a}^{d}}_{/R}, \rho)$ as above.  Anderson~\cite{A86} showed that one has an exponential map $\Exp_{G}$,  that is an entire $\FF_{q}$-linear map 
\[\Exp_{G}:\Lie G(\CC_{\infty})\rightarrow G(\CC_{\infty}) \]
satisfying the functional equation
\[\Exp_{G}\circ \partial \rho(a)=\rho(a)\circ \Exp_{G} \ \forall a\in \FF_{q}[t]. \] As a $d$-variable vector-valued power series, it is expressed as 
\[ \Exp_{G}(\bz)=\sum_{i=0}^{\infty} e_{i}\bz^{(i)} ,\]
where $e_{0}=I_d$ and $e_{i}\in \Mat_{d}(F)$ for all $i$ and $\bz=(z_{1},\ldots,z_{d})^{\tr}$ with $\bz^{(i)}:=(z_{1}^{q^{i}},\ldots,z_{d}^{q^{i}})^{\tr}$. As vector-valued power series, the formal inverse of $\Exp_{G}$ is called the logarithm of $G$ and is denoted by $\Log_{G}$. That is, we have the identities
\[\Log_{G}\circ \Exp_{G}=\textnormal{identity}=\Exp_{G}\circ \Log_{G}  \]
and we find that $\Log_{G}$ satisfies the functional equation
\[\Log_{G}\circ \rho(a)=\partial \rho(a)\circ \Log_{G}\ \forall a\in \FF_{q}[t].  \]
We note that in general, $\Log_{G}$ is not an entire function on $ G(\CC_{\infty})$ but converges on a certain open subset. 

By a morphism $\phi$ of two $t$-modules  $G_{1} = ( \GG_{a/R}^{d_{1}}, \rho_{1})$ and $G_{2} = ( \GG_{a/R}^{d_{2}}, \rho_{2})$ over $R$, we mean that $\phi:G_{1}\rightarrow G_{2}$ is a morphism of algebraic group schemes over $R$ and that it commutes with the $\FF_{q}[t]$-actions. By writing $\phi=\sum_{i=0}^{\infty}\alpha_{i}\tau^{i}\in \Mat_{d_{2}\times d_{1}}(R)[\tau]$, we note that the differential of $\phi$ at the origin is identified with 
\begin{equation}\label{E: partial phi}
\partial \phi:=\alpha_{0}:\Lie G_{1}\rightarrow \Lie G_{2} .
\end{equation}
The exponential maps of $t$-modules are functorial in the sense that one has the following commutative diagram:

\[
\xymatrix{
 G_{1} \ar@{->}[r]^{ \phi}& G_{2}  \\
 \Lie G_{1} \ar@{->}[r]^{ \partial \phi} \ar[u]^{\Exp_{G_{1}}} & \Lie G_{2} \ar[u]_{\Exp_{G_{2}}} .
}\] 

\subsection{From Anderson dual $t$-motives to Anderson $t$-modules}\label{Subsec: t-motives to t-modules}

In this section we will construct Frobenius modules and dual $t$-motives and show how one can functorially construct $t$-modules from these objects. In what follows, we take an algebraically closed subfield $\KK$ of $\CC_{\infty}$ containing $K$. For example, $\KK$ can be $\oK$ or $\CC_{\infty}$.  Let $\KK[t,\sigma]:=\KK[t][\sigma]$ be the ring obtained by joining the non-commutative variable  $\sigma$ to the polynomial ring $\KK[t]$ subject to the relation
\[ \sigma f=f^{(-1)} \sigma\hbox{ for }f\in \KK[t]  .\]
Note that $\KK[t,\sigma]$ contains the two subrings $\KK[t]$ and $\KK[\sigma]$, but the latter is non-commutative. 

\subsubsection{Frobenius modules and dual $t$-motives}\label{Subsubsec:FrobeniusModules}
We follow~\cite{CPY14} to adapt the terminology of Frobenius modules.

\begin{definition}
A Frobenius module over $\KK$ is a left $\KK[t,\sigma]$-module that is free of finite rank over $\KK[t]$.
\end{definition}

The most basic example of a Frobenius module is the trivial module ${\bf{1}}$, which has underlying module $\KK[t]$, on which $\sigma$ acts by
\[ \sigma f:=f^{(-1)} \ \forall f\in{\bf{1}}. \]
Another important example is the $n$-th tensor power of the Carlitz $t$-motive $C^{\otimes n}$ for $n\in \NN$. Here, $\KK[t]$ is the underlying module of $C^{\otimes n}$, on which the action of $\sigma$ is given by
\[\sigma f:=(t-\theta)^{n}  f^{(-1)}  \ \forall f\in C^{\otimes n}  .\]

 One core object  that we study in this paper is the $t$-module arising from the following Frobenius module $M$. \begin{definition}  Fix a positive integer $r$ and an $r$-tuple of positive integers $\fs=(s_{1},\ldots,s_{r})\in \NN^{r}$ and an $r$-tuple of polynomials ${\mathfrak{Q}} = (Q_{1},\ldots,Q_{r})\in \KK[t]^r$.  Let $M$ be a free left $\KK[t]$-module of rank $r+1$ with a fixed basis $\left\{m_{1},\ldots,m_{r+1} \right\}$ and put ${\bf{m}}:=(m_{1},\ldots,m_{r+1})^{\tr}\in \Mat_{(r+1)\times 1}(M)$. We define the following matrix
\begin{equation}\label{E:Phi s}
\Phi :=
               \begin{pmatrix}
                (t-\theta)^{s_{1}+\cdots+s_{r}}  & 0 & 0 &\cdots  & 0 \\
                Q_{1}^{(-1)}(t-\theta)^{s_{1}+\cdots+s_{r}}  & (t-\theta)^{s_{2}+\cdots+s_{r}} & 0 & \cdots & 0 \\
                 0 &Q_{2}^{(-1)} (t-\theta)^{s_{2}+\cdots+s_{r}} &  \ddots&  &\vdots  \\
                 \vdots &  & \ddots & (t-\theta)^{s_{r}} & 0 \\
                 0 & \cdots & 0 & Q_{r}^{(-1)}(t-\theta)^{s_{r}} & 1 \\
               \end{pmatrix}\in \Mat_{(r+1)}(\KK[t]),
\end{equation}
then define a left $\KK[t,\sigma]$-module structure on $M$ by setting
\[ \sigma {\bf{m}}:=\Phi {\bf{m}} .\]

\end{definition}

It follows from the definition that $M$ is a Frobenius module.  We mention that this $M$ was first studied by Anderson-Thakur~\cite{AT09} in order to give a period  interpretation for the multiple zeta value $\zeta_{A}(\fs)$ (defined in \eqref{MZVdef}) when restricting $Q_{i}$ to the Anderson-Thakur polynomial $H_{s_{i}-1}$.  It was then revisited in \cite{C14, CPY14} for studying Carlitz multiple polylogarithms when restricting $Q_{i}$ to certain algebraic elements over $K$. 

We let $M'$ be the Frobenius submodule of $M$ which is the free $\KK[t]$-submodule of rank $r$ spanned by the basis $\left\{ m_{1},\ldots,m_{r}\right\}
$, where the action of $\sigma$ on
\begin{equation}\label{m'basis}
\bmm' := \left( m_{1},\ldots,m_{r}\right)^{\tr}\in \Mat_{r\times 1}(M)
\end{equation}
is represented by the matrix
\begin{equation}\label{E:Phi s'}
\Phi' :=
\begin{pmatrix}
                (t-\theta)^{s_{1}+\cdots+s_{r}}  &  &  &  \\
                Q_{1}^{(-1)}(t-\theta)^{s_{1}+\cdots+s_{r}}  & (t-\theta)^{s_{2}+\cdots+s_{r}}   &  &  \\
                  & \ddots & \ddots &  \\
                  &  & Q_{r-1}^{(-1)}(t-\theta)^{s_{r-1}+s_{r}} & (t-\theta)^{s_{r}}  \\
\end{pmatrix}\in \Mat_{r}(\KK[t]).
\end{equation}
Note that  $\Phi'$ is the square matrix of size $r$ cut from the upper left square of $\Phi$. Following~\cite{CM17a}, for each $1\leq i\leq r$ we put
\begin{equation}\label{E: d_i}
d_{i}:=s_{i}+\cdots+s_{r}  .
\end{equation}
One observes that $M'$ possesses the following properties (cf.~\cite{CPY14}):
\begin{itemize}
\item $M'$ is free of rank $r$ over $\KK[t]$.
\item $M'$ is free of rank $d:=d_{1}+\cdots+d_{r}$ over $\KK[\sigma]$.
\item $(t-\theta)^{n} M'\subset \sigma M'$ for all integers $n\geq d_{1}$ (see the proof of Proposition~\ref{P: delta0}).
\end{itemize}
Note that a natural $\KK[\sigma]$-basis of $M'$ is given by
\begin{equation}\label{sigmabasis}
\left\{ (t-\theta)^{d_{1}-1}m_{1},(t-\theta)^{d_{1}-2}m_{1}, \ldots,m_{1},\ldots,(t-\theta)^{d_{r}-1}m_{r}, (t-\theta)^{d_{r}-2}m_{r}, \ldots,m_{r}  \right\}
\end{equation}
and we label this basis as $\left\{ e_{1},\ldots,e_{d} \right\}$. Note further that $M'$ is a dual $t$-motive in the sense of \cite[Sec.~4.4.1]{ABP04}. 

\begin{definition}\label{D:dualtmotive}
A dual $t$-motive is a left $\KK[t,\sigma]$-module $\cM$ with the following three properties.
\begin{itemize}
\item $\cM$ is free of finite rank over $\KK[t]$.
\item $\cM$ is free of finite rank over $\KK[\sigma]$.
\item $(t-\theta)^{n} \cM \subset \sigma \cM$ for all $n \gg 0$.
\end{itemize}
\end{definition}

\subsubsection{The $t$-module associated to $M'$}\label{Subsec: t-module G} In this section, we quickly review Anderson's theory of $t$-frames, which allows one to construct the $t$-module $G:=({\GG_{a}^{d}}_{/ \KK},\rho)$ which is associated to the Frobenius module $M'$. Since $M'$ is free over $\KK[t]$ with basis $\left\{m_{1},\ldots,m_{r} \right\}$, we can identify $\Mat_{1\times r}(\KK[t])$ with $M'$:
\[ 
\begin{array}{rcl}
\Mat_{1\times r}(\KK[t])&\rightarrow& M'\\
(a_{1},\ldots,a_{r})&\mapsto& a_{1}m_{1}+\cdots+a_{r}m_{r} .
\end{array} \]
As $M'$ is also free over $\KK[\sigma]$ with basis $\left\{ e_{1},\ldots,e_{d}\right\}$, we can identify $M'$ with $\Mat_{1\times d}(\KK[\sigma])$:
\[
\begin{array}{ccl}
 M'&\rightarrow &\Mat_{1\times d}(\KK[\sigma]) \\
b_{1}e_{1}+\cdots+b_{d}e_{d}& \mapsto & (b_{1},\ldots,b_{d}).
\end{array}
\] Composing the two maps above, we have the following identification
\[ 
\begin{array}{rrcl}
\iota : & \Mat_{1\times r}(\KK[t])& \rightarrow & \Mat_{1\times d}(\KK[\sigma])\\
& (a_{1},\ldots,a_{r}) &\mapsto  & (b_{1},\ldots,b_{d})
\end{array}
 \]by expressing elements of $M'$ in terms of the fixed $\KK[t]$-basis and $\KK[\sigma]$-basis above.

We remark that if $x \in M'$ can be written as $x=a_1m_1 + \dots + a_rm_r$ as above  and ${\bf{m}}':=(m_{1},\ldots,m_{r})^{\tr}\in \Mat_{r \times 1}(M')$, then we have the equation
\[  \sigma x = \sigma (a_1,\dots,a_r){\bf{m}}' = (a_1,\dots,a_r)\twistinv \Phi' {\bf{m}}', \]
and thus under the identification $\Mat_{1\times r}(\KK[t])\rightarrow M'$ the action of $\sigma$ on $\Mat_{1\times r}(\KK[t])$ is given by
\begin{equation}\label{sigmaaction}
\sigma(a_1,\dots,a_r) = (a_1,\dots,a_r)\twistinv  \Phi'.
\end{equation}
We similarly observe that the action of $\sigma$ on $\Mat_{1\times d}(\KK[\sigma])$ is the diagonal action.

Next, we define maps $\delta_0,\delta_1: \Mat_{1\times d}(\KK[\sigma])= \Mat_{1\times d}(\KK)[\sigma]\to \KK^d$ by
\begin{equation}\label{E: Def of delta0}
\delta_0\left (\sum_{i\geq 0} \bc_i \sigma^i\right ) := \bc_0^{\tr}
\end{equation}
\begin{equation}\label{E: Def of delta1}
\delta_1\left (\sum_{i\geq 0} \bc_i \sigma^i\right ) = \delta_1\left (\sum_{i\geq 0} \sigma^i \bc_i^{(i)} \right ) := \sum_{i\geq 0}\left (\bc_i\twisti\right )^{\tr}.
\end{equation}
Note that in~\cite{CPY14} we denote by  $\Delta:M'\twoheadrightarrow \KK^{d}$ the $\FF_{q}$-linear composite map  of $M'\cong \Mat_{1\times d}(\KK[\sigma])$ and $\delta_{1}$. We then have $\Ker \Delta=(\sigma-1)M'$, and so have the identification as $\FF_{q}$-vector spaces
\begin{equation}\label{Amotivediagram}
\begin{diagram}
M'/(\sigma-1)M' &\rTo^{\Delta} & \KK^{d}  \\
\dTo^{a} & & \dTo_{\,\rho(a)} \\
M'/(\sigma-1)M' &\rTo^{\Delta} & \KK^{d}
\end{diagram}.
\end{equation}
Since we have an $\FF_{q}[t]$-module structure on $M'/(\sigma-1)M'$ given by multiplication, we can equip an induced left $\FF_{q}[t]$-module structure on $\KK^{d}$ via the identification above, which we denote as $\rho$.  As the $\KK$-valued points of ${\GG_{a}^{d}}_{/ \KK}$ is $\KK^{d}$ which is Zariski dense inside the algebraic group ${\GG_{a}^{d}}_{/ \KK}$, the $\FF_{q}[t]$-module structure on $\KK^{d}$ above gives rise to an $\FF_{q}$-linear ring homomorphism
\[ \rho: \FF_{q}[t]\rightarrow \Mat_{d}(\KK)[\tau]. \]
This defines the $t$-module associated to $M'$, which we will denote by $G:=({\GG_{a}^{d}}_{/ \KK},\rho)$. For more details on this construction, see \cite[\S 5.2]{HJ16}.

\subsection{One crucial result of Anderson} Now we take the field $\KK=\CC_{\infty}$, and consider the $\FF_{q}$-linear map
\[ \delta_{0}\circ \iota: \Mat_{1\times r}(\CC_{\infty}[t]) \rightarrow \CC_{\infty}^{d}. \]
In order to give an explicit description of the above map, we make the following definition first.

\begin{definition}
Given a function $f(t) \in \TT_\theta$ which has a Taylor series centered at $\theta$ that is given by
\[ f(t) =  \sum_{n=0}^{\infty} c_{n}(t-\theta)^{n}, \]
we define the $k$-th jet of $f$ at the point $\theta$ to be the polynomial 

\[J^k_\theta(f) = c_k(t-\theta)^k+\cdots+c_1(t-\theta)+c_0  \in \C_\infty[t].\]
\end{definition}

Note that by~(\ref{E: eta}), power series in $\TT_\theta$ always have such a Taylor series as given above, and so our definition is natural.  We now describe the explicit formula for the map $\delta_0\circ \iota$.

\begin{proposition}\label{P: delta0}
For a vector $\ba = (a_1, \dots , a_r) \in \Mat_{1\times r}(\CC_{\infty}[t])$, for each coordinate $a_{i}$, we form the $(d_i - 1)$-st jet at $\theta$ and label these as
\[ J^{d_i-1}_\theta (a_{i}) = c_{i,1}(t-\theta)^{d_{i}-1}+c_{i,2}(t-\theta)^{d_{i}-2}+\cdots+c_{i,d_{i}}.\]
Then the map $\delta_0 \circ \iota$ is given by 
\[ 
\delta_0\circ \iota(\ba) = 
(c_{1,1}, c_{1,2}, \ldots, c_{1,d_{1}}, c_{2,1}, c_{2,2}, \ldots, c_{2,d_{2}}, \ldots, c_{r,1}, c_{r,2}, \ldots, c_{r,d_{r}})^{\tr}. 
\]
\end{proposition}

\begin{proof} 
By the definition of $\delta_{0}$, it suffices to show for each $i$ with $1\leq i \leq r$ that, $(t-\theta)^{N} m_{i} \in \sigma M'$  for all $N \geq d_{i}$. We prove this assertion by induction on $i$.
For $i = 1$, the claim holds since we have
\begin{align*}
(t-\theta)^{N} m_{1} = (t-\theta)^{N-d_{1}} \cdot (t-\theta)^{d_{1}} m_{1}
= (t-\theta)^{N-d_{1}} \cdot \sigma m_{1}
= \sigma (t-\theta^{q})^{N-d_{1}} m_{1} \in \sigma M'.
\end{align*}
Let $i \geq 2$, and assume that there exists $m \in M'$ such that $(t-\theta)^{d_{i-1}} m_{i-1} = \sigma m$.
Then we have
\begin{align*}
(t-\theta)^{N} m_{i} &= (t-\theta)^{N-d_{i}} \cdot (t-\theta)^{d_{i}} m_{i}
= (t-\theta)^{N-d_{i}} \left( \sigma m_i - Q_{i-1}\twistinv (t-\theta)^{d_{i-1}} m_{i-1} \right) \\
&= (t-\theta)^{N-d_{i}} \left( \sigma m_{i} - Q_{i-1}\twistinv \sigma m \right)
= \sigma (t-\theta^{q})^{N-d_{i}} \left( m_{i} - Q_{i-1} m \right) \in \sigma M'.
\end{align*}
\end{proof}

In other words, the map $\delta_{0}\circ \iota$ factors through the following map still denoted by $\delta_{0}\circ \iota$:
\[ \CC_{\infty}[t]/\left((t-\theta)^{d_{1}}\right)\times \cdots\times \CC_{\infty}[t]/\left((t-\theta)^{d_{r}}\right)\rightarrow \CC_{\infty}^{d}. \]
Anderson gives a theorem (see~\cite{HJ16} and~\cite{NP18}) that states that there exists a unique extension of the composition $\delta_0\circ \iota$ to vectors over the Tate algebra $\TT_\theta$ and he calls this extension 
\[\widehat{\delta_0\circ \iota}:\Mat_{1\times r}(\TT_\theta) \to \CC_{\infty}^d. \]
 With the above analysis of $\delta_0\circ \iota$ we can see concretely that the procedure for calculating the extended map $\widehat{\delta_0\circ \iota}$ is the same as for the original.  Namely, $\widehat{\delta_0\circ \iota}$ is given by composing the following maps
 \[ \Mat_{1\times r}(\TT_{\theta})\hookrightarrow \Mat_{1\times r}(\power{\CC_{\infty}}{t-\theta})  \twoheadrightarrow  \prod_{i=1}^{r} \power{\CC_{\infty}}{t-\theta}/\left( (t-\theta)^{d_{i}}\right) \cong  \prod_{i=1}^{r} \CC_{\infty}[t]/\left( (t-\theta)^{d_{i}}\right) \xrightarrow[]{\delta_{0}\circ \iota} \CC_{\infty}^{d}, \]
 where the first embedding is via the following natural embedding componentwise
 \begin{equation}\label{E: eta}
 \begin{array}{rccl}
\eta: &\TT_{\theta} &\hookrightarrow& \power{\CC_{\infty}}{t-\theta} \\
&&&\\
&{\displaystyle \sum_{i=0}^{\infty}b_{i} t^{i}}& \mapsto &{\displaystyle \eta(\sum_{i=0}^{\infty}b_{i} t^{i}):= \sum_{i=0}^{\infty}\left( \sum_{j=i}^{\infty} {{j}\choose {i}}  b_{j} \theta^{j-i}  \right) (t-\theta)^{i}}. 
 \end{array}
 \end{equation}
 Note that $\eta$ is injective. Indeed, assume that $\eta({\displaystyle \sum_{i=0}^{\infty}}b_{i} t^{i}) = 0$, and fix $i_{0} \geq 0$. Since we have ${\displaystyle \sum_{j=i_{0}}^{\infty} {{j}\choose {i_{0}}}} b_{j} \theta^{j-i_{0}} = 0$ and the absolute value $|\cdot|_{\infty}$ is non-archimedean, there exists $i_{1} > i_{0}$ such that
\[{ |b_{i_{0}} |_{\infty} \leq \left|{{i_{1}}\choose {i_{0}}} b_{i_{1}} \theta^{i_{1}-i_{0}} \right|_{\infty} }\]
and so
\[ |b_{i_{0}} \theta^{i_{0}}|_{\infty} \leq \left|{{i_{1}}\choose {i_{0}}} b_{i_{1}} \theta^{i_{1}} \right|_{\infty} \leq |b_{i_{1}} \theta^{i_{1}}|_{\infty}. \]
Repeating this argument, we can take $i_{0} < i_{1} < i_{2} < \cdots$ such that $|b_{i_{0}} \theta^{i_{0}}|_{\infty} \leq |b_{i_{1}} \theta^{i_{1}}|_{\infty} \leq |b_{i_{2}} \theta^{i_{2}}|_{\infty} \leq \cdots$. Since $|b_{i} \theta^{i}|_{\infty} \to 0$, we have $b_{i_{0}} = 0$. Therefore, $\eta$ is injective.
Moreover, we can verify that the image of $\eta$ is
\[ \left \{ \sum_{i=0}^\infty b_i (t-\theta)^i \in \power{\CC_{\infty}}{t-\theta} \biggm| \big\lvert b_i \theta^i \big\rvert_{\infty} \to 0 \right \}, \]
and we can construct $\eta^{-1}$ directly. 

Since it does not cause any confusion, we will drop the hat notation and just use $\delta_0\circ \iota$ for the extension. We now state Anderson's theorem from his unpublished notes, which one can find the statement and its proof in~\cite[Cor.~2.5.23]{HJ16} and~\cite{NP18}.

\begin{theorem}[Anderson]\label{T: Anderson}
Let $\Phi'$ define the dual $t$-motive given in \eqref{E:Phi s'}, and let its associated $t$-module be denoted by $G$.  For all $\bg \in \Mat_{1\times r}(\TT_\theta)$ and $\bw \in \Mat_{1\times r}(\C_\infty[t])$ satisfying the functional equation
\begin{equation}\label{AndersonCondition}
\bg\invtwist \Phi' - \bg = \bw
\end{equation}
one has
\begin{equation}\label{AndersonTheorem}
\Exp_G \left(\delta_0 \circ \iota(\bg+\bw) \right)=\delta_{1}\circ \iota(\bw).
\end{equation}
\end{theorem}

\section{Hyperderivatives of $t$-motivic Carlitz multiple star polylogarithms}\label{Sec: Hyperderivatives of CMSPL}
Carlitz multiple polylogarithms (CMPL's) were introduced by the first author in \cite{C14} and are generalizations of the Carlitz polylogarithm of Anderson and Thakur from \cite{AT90}.  Carlitz multiple star polylogarithms (CMSPL's) were introduced by the first and third authors in~\cite{CM17b} in order to connect MZV's with logarithms of $t$-modules.  The primary goal of this section is to derive the formula stated in Theorem~\ref{T: IntrodT1}, which connects the logarithmic vector \eqref{D:logvector} with these CMSPL's.

\subsection{$t$-motivic Carlitz multiple star polylogarithms}\label{Sub: t-motivic CMSPL}
  For $\fs = (s_1,\dots,s_r) \in \N^r$, we define the associated ($r$-variable) CMPL and CMSPL by
\begin{equation}\label{Lidef}
\Li_\fs(z_1,\dots,z_r) = \sum_{i_1>\dots>i_r \geq 0} \frac{z_1^{q^{i_1}}\dots z_r^{q^{i_r}}}{L_{i_1}^{s_1}\dots L_{i_r}^{s_r}}\in \power{K}{z_{1},\ldots,z_{r}}
\end{equation}
and
\begin{equation}\label{Lisdef}
\Lis_\fs(z_1,\dots,z_r) = \sum_{i_1\geq \dots \geq i_r \geq 0} \frac{z_1^{q^{i_1}}\dots z_r^{q^{i_r}}}{L_{i_1}^{s_1}\dots L_{i_r}^{s_r}}\in \power{K}{z_{1},\ldots,z_{r}}.
\end{equation}
Also the  $t$-motivic CMPL and $t$-motivic CMSPL are defined by
\begin{equation}\label{tLidef}
\fLi_\fs(t;z_1,\dots,z_r) = \sum_{i_1>\dots>i_r \geq 0} \frac{z_1^{q^{i_1}}\dots z_r^{q^{i_r}}}{\LL_{i_1}^{s_1}\dots \LL_{i_r}^{s_r}}\in \power{K}{t,z_{1},\ldots,z_{r}}
\end{equation}
and
\begin{equation}\label{tLisdef}
\fLis_{\fs}(t;z_1,\dots,z_r) = \sum_{i_1\geq \dots \geq i_r \geq 0} \frac{z_1^{q^{i_1}}\dots z_r^{q^{i_r}}}{\LL_{i_1}^{s_1}\dots \LL_{i_r}^{s_r}}\in \power{K}{t,z_{1},\ldots,z_{r}},
\end{equation}
and observe that $\fLi_\fs|_{t=\theta} = \Li_\fs$ and $\fLis_\fs|_{t=\theta} = \Lis_\fs$.

Given $\bu=(u_{1},\ldots,u_{r})^{\tr}\in \oK^{r}$, we note that  $\dfrac{u_1^{q^{i_1}} \cdots u_r^{q^{i_r}}}{\LL_{i_1}^{s_1} \cdots \LL_{i_r}^{s_r}} \in \TT_{\theta}$, and if
\begin{equation}\label{E: hypothesis on u}
|u_1|_{\infty} < q^{\frac{s_1 q}{q-1}} \  {\textnormal{and}} \ |u_i|_{\infty} \leq q^{\frac{s_i q}{q-1}} \ \mathrm{for \ each} \ 2 \leq i \leq r,
\end{equation}
then  $\left \lVert \dfrac{u_1^{q^{i_1}} \cdots u_r^{q^{i_r}}}{\LL_{i_1}^{s_1} \cdots \LL_{i_r}^{s_r}} \right \rVert_{\theta} \to 0, \ (i_1 \to \infty)$. Thus $\fLi_{\fs}(t; u_1, \dots, u_r)$ and $\fLis_{\fs}(t; u_1, \dots, u_r)$ converge in $\TT_{\theta}$ when specializating $z_{i}=u_{i}$ for each $i$, and in this situation we have 
\[ \fLi_{\fs}(\theta; u_{1},\ldots,u_{r})=\Li_{\fs}(\bu)\hbox{ and }\fLis_{\fs}(\theta;u_{1},\ldots,u_{r})=\Lis_{\fs}(\bu).
\]
In this situation, we simplify the notation by putting

\[ \fLi_{\fs,\bu}(t):=\fLi_{\fs}(t;u_{1},\ldots,u_{r})\in \power{\oK}{t} \hbox{ and } \fLis_{\fs,\bu}(t):=\fLis_{\fs}(t;u_{1},\ldots,u_{r})\in \power{\oK}{t}. \]

When $r=1$ and $s=1$, the series $\fLi_{1,u}(t)$ for $u\in \overline{K}^{\times}$ with $|u|_{\infty}<q^{\frac{q}{q-1}}$ was introduced by Papanikolas~\cite{P08} to relate the Carlitz logarithm at $u$ to a period of certain $t$-motive.  It was then generalized by the first author and Yu~\cite{CY07} in the case of $r=1$ and $s>1$ to study Carlitz zeta values. The general series $\fLi_{\fs,\bu}$ was further studied by the first author~\cite{C14} to relate the CMPL's at algebraic points to periods of certain $t$-motives. These series play an essential role when applying the ABP-criterion~\cite{ABP04} for the CMPL's at algebraic points in question. See also~\cite{CPY14, M17}.

\subsection{Hyperderivatives}\label{SS:hyperderivatives}
The main formula in this section expresses the logarithmic vector \eqref{D:logvector} in terms of hyperderivatives of $t$-motivic CMSPL's.  In this section we outline some of the basic theory for hyperderivatives. For any non-negative integer $n$, we define the $n$-th hyperderivative (with respect to $t$)  $\partial_{t}^{n}:\laurent{\CC_{\infty}}{t}\rightarrow \laurent{\CC_{\infty}}{t}$
by

\begin{equation}\label{E: DefHyperderivative}
\partial_{t}^{n} \left( \sum_{i=i_{0}}^{\infty} a_{i}t^{i} \right):=\sum_{i=i_{0}}^{\infty} {{i}\choose{n}}a_{i}t^{i-n},  
\end{equation}
where ${{i}\choose{n}}$ refers to the usual binomial coefficient, but modulo $p$. From the definition one sees that $\partial_{t}^{n}$ is a $\CC_{\infty}$-linear operator and that $\partial_{t}^{0}$ is the identity map. We further note that the hyperderivatives  satisfy the product rule: for $n\in \NN$ and $f,g\in \laurent{\CC_{\infty}}{t}$,
\begin{equation}\label{E:ProdRule}
\partial_{t}^{n}\left( fg\right)=\sum_{i=0}^{n} \partial_{t}^{i} (f)\cdot \partial_{t}^{n-i}(g).
\end{equation}

 In this paper, we are interested in Taylor coefficients of the series expansion of $f\in \TT_{\theta}$ at $t=\theta$, and the following proposition shows that such Taylor coefficients are expressed as the hyperderivatives of $f$ evaluated at $t=\theta$ (see~\cite[Lem.~2.4.1]{Pp} in the case of rational functions and \cite[Cor.~2.7]{US98} for the general case).

\begin{proposition}\label{P: Taylor coeff}
For any $f\in \TT_{\theta}$, we write $\eta(f)=\sum_{i=0}^{\infty} a_{i}(t-\theta)^{i}$. Then for any non-negative integer $n$ we have
\[\partial_{t}^{n}(f) \in \TT_{\theta} \ \mathrm{and} \ a_{n}=\partial_{t}^{n}(f)|_{t=\theta}  .\]
\end{proposition}

\begin{remark}
Since under the hypothesis~(\ref{E: hypothesis on u}) the above series $\fLi_{\fs,\bu}$ and $\fLis_{\fs,\bu}$ are in $\TT_{\theta}$, the hyperderivatives $\partial_{t}^{j}\fLi_{\fs,\bu} $ and  $\partial_{t}^{j}\fLis_{\fs,\bu}$ are still in $\TT_{\theta}$ for every positive integer $j$, whence we can specialize at $t=\theta$.
\end{remark}

\subsection{Statement of the formulae}

From now on, we fix two $r$-tuples $\fs=(s_{1},\ldots,s_{r})\in \NN^{r}$ and $\bu=(u_{1},\ldots,u_{r})\in \oK^{r}$.

\subsubsection{$t$-modules associated to CMSPL's}\label{Sub:Gsu} We note that the $t$-module $G=(\GG_{a}^{d},\rho)$ associated to the dual $t$-motive $M'$ in Sec.~\ref{Subsec: t-module G} when replacing $Q_i$ by $u_{i}$ can be explicitly written down. We will use this explicit description of $G$ heavily going forward, so we take a moment to recall it from \cite{CM17a}. Let $d_{i}$ be given in \eqref{E: d_i} for $1\leq i\leq r$, and put $d:=d_{1}+\cdots+d_{r}$. Let $B$ be a $d \times d$-matrix of the form

\[
\left( \begin{array}{c|c|c}
B[11] & \cdots & B[1r] \\ \hline
\vdots & & \vdots \\ \hline
B[r1] & \cdots & B[rr]
\end{array} \right),
\]
where $B[\ell m]$ is a $d_{\ell} \times d_{m}$-matrix for each $\ell$ and $m$ and we call $B[\ell m]$ the $(\ell, m)$-th block sub-matrix of $B$.

For $1 \leq \ell \leq m \leq r$, we define the following matrices:
\begin{enumerate}
\item $N_{\ell} \in \Mat_{d_{\ell}}(\FF_{q})$ is the matrix with $1$'s along the first super diagonal and $0$'s elsewhere.
\item $N \in \Mat_{d}(\FF_{q})$ is the block diagonal matrix with $N_\ell$ along the diagonal and $0$ blocks elsewhere.
\item $E[\ell m] \in \Mat_{d_{\ell} \times d_{m}}(\oK)$ is the matrix with a $1$ in the lower left coordinate if $\ell = m$ and with $(-1)^{m-\ell} \prod_{e=\ell}^{m-1} u_{e}$ in the lower left coordinate if $\ell <m$ and $0$'s elsewhere ($E[\ell m]=0$ otherwise).
\item $E \in \Mat_{d}(\oK)$ is the upper triangular block matrix with $E[\ell m]$ in each block.
\end{enumerate}

We then define the $t$-module $G = G_{\fs, \bu} := (\GG_{a}^{d}, \rho)$ by
\begin{equation}\label{E:Explicit t-moduleCMPL}
  \rho(t) = \theta I_{d} + N + E \tau
  \in \Mat_{d}(\oK[\tau]),
\end{equation}
and note that $G$ depends  only on $u_{1},\ldots,u_{r-1}$. Finally, we define the special point
\begin{equation}\label{E:v_s,u} 
\begin{array}{rrcccccccccccccl}
& & \multicolumn{4}{c}{\overbrace{\hspace{10.0em}}^{d_{1}}} & \multicolumn{4}{c}{\overbrace{\hspace{10.0em}}^{d_{2}}} & & \multicolumn{4}{c}{\overbrace{\hspace{4.5em}}^{d_{r}}} & \\
\bv = \bv_{\fs, \bu} := \hspace{-0.5em} & \ldelim({1}{0em} & 0, \hspace{-0.6em} & \ldots, \hspace{-0.6em} & 0, \hspace{-0.6em} & (-1)^{r-1} u_{1} \cdots u_{r}, \hspace{-0.6em} & 0, \hspace{-0.6em} & \ldots, \hspace{-0.6em} & 0, \hspace{-0.6em} & (-1)^{r-2} u_{2} \cdots u_{r}, \hspace{-0.6em} & \ldots, \hspace{-0.6em} & 0, \hspace{-0.6em} & \ldots, \hspace{-0.6em} & 0, \hspace{-0.6em} & u_{r} & \hspace{-0.6em} \rdelim){1}{0.5em}^{\tr} \in G(\oK).
\end{array} 
\end{equation}
It is not hard to see that either the $t$-module $G$ is $\bC^{\otimes d_{1}}$ if $r=1$, or that $G$ is an iterated extension of the tensor powers of the Carlitz module if $r>1$.   Finally, we note that if we take $\bu=(u_{1},\ldots,u_{r})\in A^{r}$, then the $t$-module $G$ is defined over $A$ in the sense that $\rho(t)\in \Mat_{d}(A[\tau])$, and $\bv\in G(A)$. 

\begin{remark}\label{Rem: Ext1}
Let $\cF$ be the category of Frobenius modules over $\oK$ with morphisms given by left $\oK[t,\sigma]$-module homomorphisms. One then sees that $M$ is an extension of ${\bf{1}}$ by $M'$, i.e., $M\in \Ext_{\cF}^{1}\left( {\bf{1}},M' \right)$. We can equip an $\FF_{q}[t]$-module structure on $ \Ext_{\cF}^{1}\left( {\bf{1}},M' \right)$ and  have the following isomorphisms as $\FF_{q}[t]$-modules due to Anderson (see~\cite{CPY14}):
\[\Ext_{\cF}^{1}\left( {\bf{1}},M' \right) \cong M'/(\sigma-1)M' \cong G(\oK).  \]
The special point $\bv$  is the image of $M$ under the composite of the isomorphisms above.  For details, see \cite{CPY14, CM17a}.
\end{remark}

\subsubsection{The explicit formulae}

In \cite{CM17b}, CMSPL's are related to certain coordinates of the logarithm of $G$ evaluated at $\bv$ under certain assumptions on the absolute values of the coordinates of $\bu$. We first recall the result as follows.

\begin{theorem}[{\cite[Thm.~4.2.3]{CM17b}}]\label{T: Thm4.2.3 of CM17}
Given any $\fs=(s_{1},\ldots,s_{r})\in \NN^{r}$, we let $\bu=(u_{1},\ldots,u_{r})\in \oK^{r}$ with $|u_i|_{\infty} \leq q^{\frac{s_i q}{q-1}}$ for each $1 \leq i \leq r-1$ and $|u_r|_{\infty} < q^{\frac{s_r q}{q-1}}$.  Let $G$ and $\bv$ be defined in \eqref{E:Explicit t-moduleCMPL} and \eqref{E:v_s,u} respectively using $\fs$ and $\bu$. Then $\Log_{G}$ converges $\infty$-adically at $\bv$ and we have the formula
\[
\Log_{G} (\bv) =
\begin{array}{rcll}
\ldelim( {15}{4pt}[] & * & \rdelim) {15}{4pt}[] & \rdelim\}{4}{10pt}[$d_{1}$] \\
& \vdots & & \\
& * & & \\
& (-1)^{r-1}\Li_{(s_{r}, \dots, s_{1})}^{\star}(u_{r}, \dots, u_{1}) & & \\
& * & & \rdelim\}{4}{10pt}[$d_{2}$] \\
& \vdots & & \\
& * & & \\
& (-1)^{r-2}\Li_{(s_{r}, \dots, s_{2})}^{\star}(u_{r}, \dots, u_{2}) & & \\
& \vdots & & \vdots \\
& * & & \rdelim\}{4}{10pt}[$d_{r}$] \\
& \vdots & & \\
& * & & \\
& \Li^{\star}_{s_{r}}(u_{r}) & & \\[10pt]
\end{array} \ \ \ \in \Lie G(\CC_{\infty}).
\] In particular, the $(s_{1}+\cdots+s_{r})$-th coordinate of $\Log_{G}(\bv)$ is $(-1)^{\dep(\fs)-1} \Lis_{\widetilde{\fs}}(\widetilde{\bu})$.
\end{theorem}

The primary result in this section is to give explicit formulae for the (previously unknown) $*$-coordinates in the theorem above. 

\begin{theorem}\label{T:Logcoords}
Let the notation and assumptions be given as in Theorem~\ref{T: Thm4.2.3 of CM17}.  For each $1\leq i\leq r$, we let $d_{i}:=s_{i}+\cdots+s_{r}$ and set $d:=d_{1}+\cdots+d_{r}$. Define
\[
\bY_{\fs,\bu}:=\begin{pmatrix}
Y_{1}\\
\vdots\\
Y_{r}
\end{pmatrix}\in \Mat_{d\times 1}(\CC_{\infty}),
\]
where for each $1\leq i\leq r$, $Y_{i}$ is given by
\[
\begin{array}{rl}
Y_{i} & = \left(
\begin{matrix}
(-1)^{r-i} \left(\partial_{t}^{d_{i}-1} \fLis_{(s_{r},\ldots,s_{i}), (u_{r},\ldots,u_{i})}(t)\right)|_{t=\theta}\\
(-1)^{r-i} \left(\partial_{t}^{d_{i}-2} \fLis_{(s_{r},\ldots,s_{i}), (u_{r},\ldots,u_{i})}(t)\right)|_{t=\theta}\\
\vdots\\
(-1)^{r-i} \left(\partial_{t}^{0} \fLis_{(s_{r},\ldots,s_{i}), (u_{r},\ldots,u_{i})}(t)\right)|_{t=\theta}
\end{matrix}
\right)\\
&\\
&=\left(
\begin{matrix}
(-1)^{r-i} \left(\partial_{t}^{d_{i}-1} \fLis_{(s_{r},\ldots,s_{i}), (u_{r},\ldots,u_{i})}(t)\right)|_{t=\theta}\\
(-1)^{r-i} \left(\partial_{t}^{d_{i}-2} \fLis_{(s_{r},\ldots,s_{i}), (u_{r},\ldots,u_{i})}(t)\right)|_{t=\theta}\\
\vdots\\
(-1)^{r-i} \Lis_{(s_{r},\ldots,s_{i})} (u_{r},\ldots,u_{i})
\end{matrix}
\right)\in \Mat_{d_i \times 1}(\CC_{\infty}).
\end{array}
\]Then we get the following formula for the logarithm evaluated at $\bv$
\[
\Log_G(\bv) = \bY_{\fs,\bu}\in \Lie G(\CC_{\infty}).
\]
\end{theorem}
\subsection{Preparation for the proof} This section is devoted to the preparation of the proof of Theorem~~\ref{T:Logcoords}. For convenience, we extend the hyperderivatives~\eqref{E: DefHyperderivative} to operators on vectors with entries in $\laurent{\CC_{\infty}}{t}$. Precisely, for a positive integer $m$ and for $g_1, \dots,g_n \in \laurent{\CC_{\infty}}{t}$ we define
\begin{equation}\label{E: partial matrix}
\partial_t^m[g_1,\dots,g_n] = 
\left (\begin{matrix}
\partial_t^{m-1}(g_1) & \hdots & \partial_t^{m-1}(g_n)\\
\vdots&   & \vdots \\
 \partial_t^{1}(g_1) & \hdots &  \partial_t^{1}(g_n)\\
g_1  & \hdots & g_n
\end{matrix}\right ) \in \Mat_{m\times n}( \laurent{\CC_{\infty}}{t}),
\end{equation}
and for $h\in \laurent{\CC_{\infty}}{t}$ define
\begin{equation}\label{E: d-matrix}
d_t^m[h] = 
\left (\begin{matrix}
h &  \partial_t^{1}(h) & \partial_t^2(h) & \hdots & \partial_t^{m-1}(h)\\
0 & h &  \partial_t^{1}(h) & \hdots & \partial_t^{m-2}(h) \\
0 & 0 & h & \hdots &  \partial_t^{m-3}(h) \\
\vdots & \vdots & \vdots & \ddots & \vdots \\
0 & 0 & 0 & \hdots & h
\end{matrix}\right ) \in \Mat_{ m}( \laurent{\CC_{\infty}}{t}).
\end{equation}
These matrices are all defined in \cite[\S 2.5]{Pp} and are called $\partial$-matrices and $d$-matrices, respectively.  We collect several facts about these matrices which are proved there.

\begin{proposition}\label{P:dmatfacts}
Let $m$ be a positive integer. 
For any $h,g_1, \dots,g_n \in \laurent{\CC_{\infty}}{t}
$, the following hold.
\begin{enumerate}
\item The $d$-matrices are multiplicative,
\[d_t^m[g_1]d_t^m[g_2] = d_t^m[g_1g_2].\]
\item We can combine $d$-matrices and $\partial$-matrices as follows,
\[d_t^m[h]\partial_t^m[g_1,\dots,g_n] = \partial_t^m[hg_1,\dots,hg_n].\]
\item Viewed as maps, $ d_t^m[\cdot]:\laurent{\CC_{\infty}}{t} \to  \Mat_{m}(\laurent{\CC_{\infty}}{t})$ and $ \partial_t^m[\cdot]:\Mat_{1\times n}(\laurent{\CC_{\infty}}{t})\to \Mat_{m\times n}(\laurent{\CC_{\infty}}{t})$ are $\CC_{\infty}$-linear injections of vector spaces.
\end{enumerate}
\end{proposition}

Applying the properties above, we have the following.

\begin{proposition}\label{P: formula of partial rho}
Let $(G,\rho)$ be the $t$-module defined in \eqref{E:Explicit t-moduleCMPL}. Then for every polynomial $b(t)\in \FF_{q}[t]$, $ \partial \rho (b)$ is a block diagonal  matrix with 
\[ d_{t}^{d_{i}}[b(t)]|_{t=\theta}=
\begin{pmatrix}
b(\theta)& \left(\partial_{t}^{1}b \right)|_{t=\theta}&\cdots&\left(\partial_{t}^{d_{i}-1}b \right)|_{t=\theta} \\
&\ddots& \ddots& \vdots\\
&&\ddots&  \left(\partial_{t}^{1}b \right)|_{t=\theta} \\
&&& b(\theta)
\end{pmatrix}\in \Mat_{d_i}(\oK)
\]located at the $i$-th block along the diagonal for $i=1,\ldots,r$.
\end{proposition}
\begin{proof}
Note that by definition of $\rho$ we have
\[ 
\partial \rho(t)=
\begin{pmatrix}
\theta I_{d_{1}}+N_{1}& &\\
& \ddots& \\
& & \theta I_{d_{r}}+N_{r}
\end{pmatrix}=
\begin{pmatrix}
\partial \bC^{\otimes d_1}(t)& &\\
& \ddots& \\
& & \partial \bC^{\otimes d_r}(t)
\end{pmatrix}
\]
and also
\[ \partial \bC^{\otimes d_i}(t)= d_{t}^{d_{i}}[ t]|_{t=\theta}\]
for each $1\leq i \leq r$. Since $\partial \rho(\cdot)$ is an $\FF_{q}$-linear ring homomorphism on $\FF_{q}[t]$ and by Proposition~\ref{P:dmatfacts} so is $d_{t}^{d_{i}}[\cdot ]$ for all $1\leq i \leq r$, the desired result follows.
\end{proof}

The following lemma is a  generalization of Yu's last coordinate logarithms theory~\cite[Thm.~2.3]{Yu91} to our $t$-module $G$, and it will be used in the proof of the formulae given in the next subsection.

\begin{lemma}\label{Yu's lemma}
Fixing $\fs=(s_{1},\ldots,s_{r})\in \NN^{r}$ and $\bu=(u_{1},\ldots,u_{r})\in  \oK^{r}$, let $G$ be the $t$-module over  $\oK$ defined in {\rm{(\ref{E:Explicit t-moduleCMPL})}}.  For each $1\leq i\leq r$, we let $d_{i}:=s_{i}+\cdots+s_{r}$.  Suppose that we have two vectors 
\[Y= \begin{pmatrix}
Y_{1}\\
\vdots\\
Y_{r}
\end{pmatrix}\in  \Lie G(\CC_{\infty}) \ \textnormal{ and } V= \begin{pmatrix}
V_{1}\\
\vdots\\
V_{r}
\end{pmatrix}\in \Lie G(\CC_{\infty})
\]
with $Y_{i}, V_{i}\in \Mat_{d_{i}\times 1}(\CC_{\infty})$ for each $i$ so that 
\begin{itemize}
\item $ \Exp_{G}(Y) \in G( \oK)$ and $ \Exp_{G}(V) \in G( \oK)$;
\item the last coordinate of $Y_i$ equals the last coordinate of $V_i$ for all $1\leq i\leq r$.
\end{itemize}
Then we have that $Y=V$. 
\end{lemma}
\begin{proof}
We first note that when $r=1$, $G=\bC^{\otimes d_{1}}$, otherwise from the explicit definition $G$ is an iterated extension of certain tensor powers of the Carlitz module.  We can also see this directly as follows, by using the dual $t$-motive $M'$ whose $\sigma$-action on the basis $\bmm'$ from \eqref{m'basis} is given by $\Phi'$ (\ref{E:Phi s'}).  Note first that the $t$-module  which corresponds to $M'$ is $G$.  Then, we put $G_{1}:=\bC^{\otimes d_{1}}$  and $G_{r}:=G$ and for each $1\leq i\leq r$, we let $\Phi_{i}'$ be the square matrix of size $i$ cut off from the upper left square of $\Phi'$, let $M_{i}'$ be the dual $t$-motive whose $\sigma$-action on a fixed $\CC_{\infty}[t]$-basis is given by $\Phi_{i}'$ and let $G_{i}$ be the $t$-module associated to $M_{i}'$ as per Sec.~\ref{Subsec: t-motives to t-modules}. For each $2\leq i\leq r$ we have the exact sequence of left $\CC_{\infty}[t,\sigma]$-modules
\begin{equation}\label{Mexact}
 \xymatrix{
0 \ar[r] & M_{i-1}' \ar@{^{(}->}[r] & M_{i}' \ar@{->>}[r] & C^{\otimes d_{i}} \ar[r] &0
}.
\end{equation}

Recall that $M_{i}'/(\sigma-1)M_{i}' \cong G_{i}(\CC_{\infty})$ and $C^{\otimes d_{i}}/(\sigma-1)C^{\otimes d_{i}}\cong \bC^{\otimes d_{i}}(\CC_{\infty})$ as $\FF_{q}[t]$-modules. Since the $\FF_{q}[t]$-linear map $(\sigma-1):\C^{\otimes d_i}\rightarrow \C^{\otimes d_i}$ is injective, the snake lemma combined with  the two previously mentioned facts show  the following exact sequence of $t$-modules

 \[ \xymatrix{
0 \ar[r] & G_{i-1} \ar@{^{(}->}[r]^{\ \ \iota_{i}}& G_{i} \ar@{->>}[r]^{\pi_{i} \ \ } & \bC^{\otimes d_{i}} \ar[r] &0
}\]induced from \eqref{Mexact}, where $\pi_i$ is the projection map onto the last $d_i$ coordinates which also equals $\partial \pi_i$ (cf.~proof of \cite[Prop.~6.1.1]{CPY14}).

We prove the lemma by induction on the depth $r$. When $r=1$, we consider $Y-V$ which is mapped to an algebraic point of $\bC^{\otimes d_1}$ via $\Exp_{\bC^{\otimes d_1}}$ by the  hypotheses on $Y$ and $V$.  Since the last coordinate of $Y-V$ is zero, we have $Y=V$ by \cite[Thm.~2.3]{Yu91}.

Suppose that the result is valid for all $r\leq n-1$ for a positive integer $n\geq 2$. Now we consider the case when $r=n$. In this case, we consider the following commutative diagram
\[
\xymatrix{
0 \ar[r] & G_{r-1} \ar@{^{(}->}[r]^{\iota_{r}}& G \ar@{->>}[r]^{\pi_{r}} & \bC^{\otimes d_{r}} \ar[r] &0\\
0 \ar[r] & \Lie G_{r-1} \ar@{^{(}->}[r]^{\ \  \partial \iota_{r}} \ar[u]^{\Exp_{G_{r-1}}} & \Lie G \ar@{->>}[r]^{ \partial \pi_{r}  \ \ }\ar[u]^{\Exp_{G}} & \Lie \bC^{\otimes d_{r}} \ar[r] \ar[u]_{ \Exp_{\bC^{\otimes d_{r}}}} &0.
}\] Note that
\[ 
\partial \pi_{r}\left( Y-V\right)=Y_{r}-V_{r}=\begin{pmatrix}
*\\
\vdots\\
0
\end{pmatrix},
\] which is mapped to an algebraic point of  $\bC^{\otimes d_{r}}$ via  $\Exp_{\bC^{\otimes d_{r}}}$. It follows again  by \cite[Thm.~2.3]{Yu91} that $Y_{r}=V_{r}$. Hence the  $Y-V$ are of the form
 \[
Y-V=\begin{pmatrix}
Y_{1} - V_{1}\\
\vdots\\
Y_{r-1} - V_{r-1}\\
{\bf{0}}
\end{pmatrix}\in \Ker \partial \pi_{r}=\Lie G_{r-1}.
\]
Using the commutative diagram above, by the hypotheses on $Y$ and $V$, the vector $Y-V$ is mapped to an algebraic point of $G_{r-1}$ via  $\Exp_{G_{r-1}}$. Then, since $G_{r-1}$ is defined using the index $(s_{1},\ldots,s_{r-2},s_{r-1}+s_{r})$ of depth $r-1$ and since the last coordinates of $Y_i$ and $V_i$ are the same for all $1\leq i \leq r-1$, by induction hypothesis we obtain that $Y_i=V_i$ for all $1\leq i\leq r-1$. Combining this with $Y_r=V_r$, the desired equality $Y=V$ follows. 
\end{proof}

In order to simplify when the situation needs to be discussed separately, from now on we use the following notation.
\begin{definition}
Let $R$ be a (not necessary commutative) ring and $(a_{k})_{k}$ a sequence in $R$.
For fixed integers $i$ and $j$, we define
\[
\prod_{i \leq k \leq j} a_{k} := \left\{ \begin{array}{cl} a_{i} a_{i+1} \cdots a_{j} & (i \leq j) \\ 1 & (i > j) \end{array} \right.
\ \mathrm{and} \
\prod_{i \leq k < j} a_{k} := \left\{ \begin{array}{cl} a_{i} a_{i+1} \cdots a_{j-1} & (i < j) \\ 1 & (i \geq j) \end{array} \right..
\] 
\end{definition}

We continue with above setting for fixed $\fs\in \NN^{r}$ and $\bu\in \oK^{r}$. Let $\Phi'$ be defined in \eqref{E:Phi s'} by letting $Q_{i}=u_{i}$ for $1\leq i\leq r$. We then define 
\[\Theta:=({\Phi'}^{-1})^{\tr}\in \Mat_{r}(\oK(t)),\]
which is an upper triangular matrix with $(i,j)$-entry equal to
\[  \Theta_{i, j} := (-1)^{j-i} \dfrac{\prod_{i \leq k < j} u_{k} \twistk{-1}}{(t - \theta)^{d_j}} \ (1\leq i \leq j\leq r).  \]
We further let $\bx_j :=  \left( (t - \theta)^{d_j-1}, (t - \theta)^{d_j-2}, \dots, 1 \right)$
and define
\begin{equation}\label{E: Def of X}
X := \left( \begin{matrix} \bx_{1} & & \\ & \ddots & \\ & & \bx_{r} \end{matrix} \right) \in \Mat_{r \times d}(A[t]).
\end{equation}

Define the following operator

\begin{equation}\label{E: Def of D}
D := \left( \begin{matrix} \partial_t^{d_1} & & \\ & \ddots & \\ & & \partial_t^{d_r} \end{matrix} \right):=\left( \left( \begin{matrix} \ba_1 \\ \vdots \\ \ba_r \end{matrix} \right) \mapsto \left( \begin{matrix} \partial_t^{d_1} [\ba_1] \\ \vdots \\ \partial_t^{d_r} [\ba_r] \end{matrix} \right) \right) \colon \Mat_{r \times d} (\oK(t)) \to \Mat_{ d} (\oK(t)), 
\end{equation}

where $\ba_i \in \Mat_{1 \times d}  (\oK(t))$ for each $i$.

The following identity is a special case of the Technical Lemma~\ref{L:technical lemma}, whose proof will be given in the next section, when replacing $Q_i$ by $u_i$ for $1\leq i\leq r$ and replacing \[\left(c_{1,1},\ldots,c_{1,d_{1}},\ldots,c_{r,1},\ldots,c_{r,d_{r}} \right)^{\tr}\] by $\bv$. It is key for the proof of Theorem~\ref{T:Logcoords} in the next subsection.

\begin{proposition}\label{P: special case of general formulae}
Given any $\fs=(s_{1},\ldots,s_{r})\in \NN^{r}$, we let $\bu=(u_{1},\ldots,u_{r}) \in \oK^{r}$ with $|u_i|_{\infty} \leq q^{\frac{s_i q}{q-1}}$ for each $1 \leq i \leq r-1$ and $|u_r|_{\infty} < q^{\frac{s_r q}{q-1}}$.  Let $G$ and $\bv$ be defined in \eqref{E:Explicit t-moduleCMPL} and \eqref{E:v_s,u} respectively using $\fs$ and $\bu$. 
 Then the following identity holds:
\[ 
\Exp_{G}\left( \sum_{n = 0}^{\infty} \left. D \left( \left( \prod_{1 \leq m \leq n} \Theta \twistk{m} \right)  X \twistk{n} \right) \right|_{t = \theta} \bv\twistk{n}\right)= \bv. 
\]
\end{proposition}

\subsection{Proof of Theorem~\ref{T:Logcoords}}
Now we give a proof of Theorem~\ref{T:Logcoords}.
Our starting point is  Proposition~\ref{P: special case of general formulae}. For each $n \geq 0$, we calculate that
\[ 
D \left( \left( \prod_{1 \leq m \leq n} \Theta \twistk{m} \right)  X \twistk{n} \right) \bv\twistk{n} =
\left( \begin{matrix} \beta_{n,1} \\ \vdots \\ \beta_{n,r} \end{matrix} \right), 
\]
where $\beta_{n,i} \in \Mat_{d_i\times 1}(\C_\infty(t))$ for $1\leq i \leq r$ is given by
\begin{align*}
\beta_{n, i} &= \sum_{j=i}^{r} \partial_t^{d_i} \left[ (-1)^{j-i} \sum_{i=k_0 \leq \cdots \leq k_n=j} \left( \prod_{1 \leq m \leq n} \dfrac{\prod_{k_{m-1} \leq k < k_{m}} u_{k}^{q^{m-1}}}{(t-\theta^{q^{m}})^{d_{k_{m}}}} \right) \cdot (-1)^{r-j} (u_j \cdots u_r)^{q^n} \right] \\
&= (-1)^{r-i} \partial_t^{d_i} \left[ \sum_{i=k_0 \leq \cdots \leq k_n \leq r} \ \dfrac{\displaystyle{\prod_{1 \leq m \leq n+1}} \ \displaystyle{\prod_{k_{m-1} \leq k < k_{m}}} u_{k}^{q^{m-1}}}{  {\displaystyle \prod_{1 \leq m \leq n} } (t-\theta^{q^{m}})^{d_{k_{m}}}  } \right] \ \ (k_{n+1} := r + 1) \\
&= (-1)^{r-i} \partial_t^{d_i} \left[ \sum_{i=k_0 \leq \cdots \leq k_n \leq r} \ \dfrac{\displaystyle{\prod_{0 \leq m \leq n}} \ \displaystyle{\prod_{k_{m} \leq k < k_{m+1}}} u_{k}^{q^{m}}}{  {\displaystyle \prod_{1 \leq m \leq m' \leq n} \ \prod_{k_{m'} \leq k < k_{m'+1}} } (t-\theta^{q^{m}})^{s_{k}}  } \right] \\
&= (-1)^{r-i} \partial_t^{d_i} \left[ \sum_{i=k_0 \leq \cdots \leq k_n \leq r} \ \dfrac{\displaystyle{\prod_{0 \leq m \leq n}} \ \displaystyle{\prod_{k_{m} \leq k < k_{m+1}}} u_{k}^{q^{m}}}{  {\displaystyle \prod_{1 \leq m' \leq n} \ \prod_{k_{m'} \leq k < k_{m'+1}} } \LL_{m'}^{s_{k}}  } \right] \\
&= (-1)^{r-i} \partial_t^{d_i} \left[ \sum_{i=k_0 \leq \cdots \leq k_n \leq r} \ \prod_{0 \leq m \leq n} \ \prod_{k_{m} \leq k < k_{m+1}} \dfrac{u_{k}^{q^{m}}}{\LL_{m}^{s_{k}}} \right].
\end{align*} 
Note that for $n \geq 0$ and $1 \leq i \leq r$, we have a bijective map of sets
\[
\left \{ (k_0, \dots, k_n) \mid i = k_0 \leq \cdots \leq k_n \leq r) \right \} \to \left \{ (m_i, \dots, m_r) \mid 0 \leq m_i \leq \cdots \leq m_r = n \right \}
\]
given by $m_{k} := \max \{ m | k_{m} \leq k \}$ and is explained by the following table. \\ \\
\begin{tabular}{|c||c|c|c|c|c|c|c|c|c|c|c|c|c|c|c|}
\hline
$k$ & $k_{0} = i$ & $\cdots$ & $k_{1}-1$ & $k_{1}$ & $\cdots$ & $k_{2}-1$ & $k_{2}$ & $\cdots$ & $k_{3}-1$ & $k_{3}$ & $\cdots$ & $k_{n}-1$ & $k_{n}$ & $\cdots$ & $r$ \\ \hline
$m_{k}$ & $0$ & $\cdots$ & $0$ & $1$ & $\cdots$ & $1$ & $2$ & $\cdots$ & $2$ & $3$ & $\cdots$ & $n-1$ & $n$ & $\cdots$ & $n$ \\ \hline
\end{tabular} \\ \\
The inverse map is given by $k_{m} := i + \# \{ k | m_{k} < m \}$.
Therefore we can express $\beta_{n,i}$ as the following:
\[
\beta_{n, i} = (-1)^{r-i} \partial_t^{d_i} \left[ \sum_{0 \leq m_i \leq \cdots \leq m_r = n} \dfrac{u_{i}^{q^{m_i}} \cdots u_r^{q^{m_r}}}{\LL_{m_i}^{s_i} \cdots \LL_{m_r}^{s_r}} \right].
\]

Thus, summing $\beta_{n,i}$ over $n\geq 0$ gives
\[
\sum_{n = 0}^\infty \beta_{n,i}  =  (-1)^{r-i} \partial_t^{d_i}  \left [ \sum_{ 0 \leq m_i \leq \dots \leq m_r} \frac{ u_i^{q^{m_i}}  \dots u_r^{q^{m_r}}}{ \LL_{m_i}^{s_i}  \dots \LL_{m_r}^{s_r}} \right ] = (-1)^{r-i} \partial_t^{d_i}  \left [ \fLis_{(s_r,\dots,s_i)}(t;u_r,\dots,u_i)\right ].
\]
 Note that the first equality comes from the continuity of of the map $\partial_{t}^{d_{i}}[\cdot] : \TT_{\theta} \to \Mat_{d_{i} \times 1}(\TT_{\theta})$,  which is clear from the definition of $\partial_{t}^{d_{i}}[\cdot]$.
To finish the proof,  we note that evaluating at $t=\theta$ gives 
\[
\sum_{n=0}^{\infty}
 \left. D\left(  \left( \prod_{1 \leq m \leq n} \Theta \twistk{m} \right)  X \twistk{n} \right) \right|_{t=\theta} \bv \twistk{n} 
=
\left .\sum_{n = 0}^\infty\left( \begin{matrix} \beta_{n,1} \\ \vdots \\ \beta_{n,r} \end{matrix} \right) \right |_{t=\theta}
=
\begin{pmatrix}
Y_{1}\\
\vdots\\
Y_{r}
\end{pmatrix}\]
and so by Proposition~\ref{P: special case of general formulae} we have
\[\Exp_{G} (Y_{1},\dots,Y_{r})^\tr =\bv. \]
Then we complete the proof by using Lemma~\ref{Yu's lemma} (since $\Exp_{G}(\Log_{G}(\bv))=\bv$)  and Theorem~\ref{T: Thm4.2.3 of CM17} that the $(d_{1} + \cdots + d_{i})$-th coordinates of $\Log_G(\bv)$ and  $\bY_{\fs,\bu}$ coincide with \[(-1)^{r-i} \Lis_{(s_{r},\ldots,s_{i})} (u_{r},\ldots,u_{i})\] for each $i$.

\begin{remark}
In the case of $r=1$, the formulae above are due to Papanikolas. See~\cite[Prop. 4.3.6]{Pp} and also \cite[(4.3.1)]{Pp}.
\end{remark}

The following is an immediate consequence of Theorem~\ref{T:Logcoords} and it will be applied in the proof of our main result, Theorem~\ref{T: IntrodT2}.

\begin{corollary}\label{C: partial bY}
 Let notation and assumptions be given in Theorem~\ref{T: Thm4.2.3 of CM17} and Theorem~\ref{T:Logcoords}. Then for any polynomial $b\in \FF_{q}[t]$, we have 
\[
\partial \rho(b) \Log_{G}(\bv)=\partial \rho(b) \bY_{\fs,\bu}=
\begin{pmatrix}
\partial\bC^{\otimes d_{1}}(b)\left( Y_{1}\right)\\
\vdots\\
\partial \bC^{\otimes d_{r}}(b)\left( Y_{r}\right)
\end{pmatrix}\in \Mat_{d}(\CC_{\infty}),
\] where for each $1\leq i\leq r$, $\partial\bC^{\otimes d_{i}}(b) \left(Y_{i}\right)$ is explicitly given by the following formula

\[
\begin{array}{rl}
\partial\bC^{\otimes d_{i}}(b) \left(Y_{i}\right) & = \left(
\begin{matrix}
(-1)^{r-i} \left(\partial_{t}^{d_{i}-1} b \fLis_{(s_{r},\ldots,s_{i}), (u_{r},\ldots,u_{i})}(t)\right)|_{t=\theta}\\
(-1)^{r-i} \left(\partial_{t}^{d_{i}-2} b \fLis_{(s_{r},\ldots,s_{i}), (u_{r},\ldots,u_{i})}(t)\right)|_{t=\theta}\\
\vdots\\
(-1)^{r-i} \left(\partial_{t}^{0} b \fLis_{(s_{r},\ldots,s_{i}), (u_{r},\ldots,u_{i})}(t)\right)|_{t=\theta}
\end{matrix}
\right)\\
&\\
&=\left(
\begin{matrix}
(-1)^{r-i} \left(\partial_{t}^{d_{i}-1} b \fLis_{(s_{r},\ldots,s_{i}), (u_{r},\ldots,u_{i})}(t)\right)|_{t=\theta}\\
(-1)^{r-i} \left(\partial_{t}^{d_{i}-2} b \fLis_{(s_{r},\ldots,s_{i}), (u_{r},\ldots,u_{i})}(t)\right)|_{t=\theta}\\
\vdots\\
(-1)^{r-i} b(\theta) \Lis_{(s_{r},\ldots,s_{i})} (u_{r},\ldots,u_{i})
\end{matrix}
\right)\in \Mat_{d_i \times 1}(\CC_{\infty}).
\end{array}
\]
\end{corollary}
\begin{proof}
The proof follows from Proposition \ref{P:dmatfacts} (2) together with Theorem~\ref{T:Logcoords} and Proposition~\ref{P: formula of partial rho}.
\end{proof}

\section{Proof of Technical Lemma}\label{Sec: general formulae}

Throughout this section, we fix a positive integer $r$, and  an index $\fs=(s_{1},\ldots,s_{r})\in \NN^{r}$ with $d_{i}$ defined in (\ref{E: d_i}). Fix $r-1$ polynomials $Q_{1},\ldots,Q_{r-1}\in \oK[t]$ and let $\Phi'\in \Mat_{r}(\oK[t])$ be the matrix~(\ref{E:Phi s'}) defined using $Q_{1},\ldots,Q_{r-1}$, and let $G$ be the $t$-module over $\oK$ associated to the dual $t$-motive $M'$ defined by $\Phi'$ in Sec.~\ref{Subsec: t-module G}. The main goal of this section is to give an explicit formula for the coordinates of the logarithm of the $t$-module $G$ in the general setting.  To do this, we apply Theorem \ref{T: Anderson} to a specially crafted function $\bg$, which culminates in Lemma \ref{L:technical lemma}.  
 
\subsection{Constructions of $\bg$ and its convergence}  The aim of this subsection is to construct a solution $\bg\in \Mat_{1\times r}(\TT_{\theta})$ of the difference equation \eqref{AndersonCondition} under certain conditions.

\subsubsection{Gauss norms} First, we define a seminorm on $\Mat_{\ell \times m}(\C_\infty)$ for $B =  (b_{i,j}) \in \Mat_{\ell \times m}(\C_\infty)$ by setting
\[ \lVert B\rVert = \max_{i,j}\left\{ | b_{i,j} |_{\infty} \right\}.\]
Note that the seminorm is only submultiplicative, i.e. for matrices $B\in \Mat_{k \times \ell}(\C_\infty)$ and $C \in \Mat_{\ell \times m}(\C_\infty)$
\[\lVert BC\rVert\leq \lVert B\rVert \cdot \lVert C\rVert.\]
The above inequality also gives
\[\lVert B\inv \rVert \geq \left (\lVert B\rVert \right )\inv \] for any invertible matrix $B$ with entries in $\CC_{\infty}$. We also have identities and inequalities for $\alpha\in \C_\infty$ and $B,C \in \Mat_{\ell \times m}(\C_\infty)$
\[\lVert \alpha B\rVert = |\alpha|_{\infty}\cdot \lVert B\rVert,\quad \lVert B+C\rVert\leq \max\left\{ \lVert B\rVert, \lVert C\rVert\right\}.\]
Then, for $\alpha \in \C_\infty^{\times}$ we define the Tate algebra
\begin{equation}\label{Tatealgdef}
   \TT_\alpha = \left \{ \sum_{i=0}^\infty b_i t^i \in \power{\C_\infty}{t} \biggm| \big\lvert b_i \alpha^i \big\rvert_{\infty} \to 0 \right \}.
\end{equation}
 If $|\alpha|_{\infty} \geq 1$, then $\TT_{\alpha}$ is stable under the action $f \mapsto f^{(n)}$ for each $n \geq 0$.
Define the Gauss norm $\lVert \cdot \rVert_\alpha$ on $\TT_{\alpha}$ by putting \[ \lVert f \rVert_{\alpha}:=\max_{i} \left\{ |b_{i}\alpha^{i}|_{\infty}  \right\} \]
  for $f=\sum_{i \geq 0} b_{i}t^{i}\in \TT_{\alpha}$. We then extend the Gauss norm to $\Mat_{\ell \times m}\left (\TT_{\alpha}\right )$ by setting
\[\lVert \bh \rVert_\alpha = \max_{i,j} \left\{ \lVert h_{ij} \rVert_{\alpha} \right\}\] 
for $\bh=(h_{ij}) \in  \Mat_{\ell,m}\left (\TT_{\alpha}\right )$.
We mention that $\lVert \bh \rVert_{\alpha}$ coincides with $\lVert \bh \rVert$ when $\bh\in \Mat_{\ell\times m}(\CC_{\infty})$. Then for $\bh\in \Mat_{k \times \ell}(\TT_\alpha)$ and $\bk \in \Mat_{\ell \times m}(\TT_\alpha)$
\[\lVert\bh\bk\rVert_\alpha\leq \lVert\bh\rVert_\alpha \cdot  \lVert\bk\rVert_\alpha.\]
 Since $\TT_{\alpha} \to \TT_{1} ; t \mapsto \alpha t$ is an isomorphism of normed algebras and $\TT_{1}$ is complete (cf.~\cite[Sec.~1.4, Prop.~3]{BGR84}), $\Mat_{\ell \times m}(\TT_\alpha)$ is complete under the Gauss norm $\lVert \cdot \rVert_{\alpha}$.

\subsubsection{Definition of $\bg$ and its convergence} 
 We fix a vector $\bw = (w_1, \dots, w_r)\in \Mat_{1\times r}(\CC_{\infty}[t])$ with
\[w_i = c_{i,1}(t-\theta)^{d_{i}-1}+c_{i,2}(t-\theta)^{d_{i}-2} +\cdots+c_{i,d_{i}}, \quad c_{i,j}\in \CC_{\infty}, \quad 1\leq i\leq r,\]
then define
\begin{equation}\label{E: Def of g}
\bg = \bw\twist (\Phi\pinv)\twist + \bw \twistk{2} (\Phi\pinv)\twistk{2}(\Phi\pinv)\twist +\bw \twistk{3} (\Phi\pinv)\twistk{3}(\Phi\pinv)\twistk{2}  (\Phi\pinv)\twist  +\cdots.
\end{equation}
A quick check shows that the general term $\bw \twistk{n} (\Phi\pinv)\twistk{n} \cdots  (\Phi\pinv)\twist $ is in $\Mat_{1\times r}(\TT_{\theta})$ and under certain hypothesis on the polynomials $Q_{i}$'s and $\bw$, it is shown to converge to zero under the Gauss norm $\lVert \cdot\rVert_{\theta}$ in the following proposition.
\begin{proposition}\label{P:Convergence} 
Assume that $\lVert Q_i \rVert_{1} \leq q^{\frac{s_i q}{q-1}}$ for each $1 \leq i \leq r-1 $. Let $\bw = (w_1, \dots, w_r)\in \Mat_{1\times r}(\CC_{\infty}[t])$ be given above with the condition that $|c_{i,j}|_{\infty} < q^{j+\frac{d_i}{q-1}}$ for each $1 \leq i \leq r$ and $1 \leq j \leq d_i$. Then the formal series $\bg$ defined in {\rm{(\ref{E: Def of g})}} converges in $\Mat_{1\times r}(\TT_{\theta})$.
\end{proposition}

\begin{proof}
Let $\mu_i := \deg_t Q_i$.
It is clear that $\lVert Q_i^{(n)} \rVert_{\theta} \leq q^{\frac{s_i q^{n+1}}{q-1} + \mu_i}$ for each $n \geq 0$.
For each $n \geq 1$, we have $\dfrac{1}{t - \theta^{q^{n}}} \in \TT_{\theta}$ and $\left\lVert \dfrac{1}{t - \theta^{q^{n}}} \right\rVert_{\theta} = q^{-q^{n}}$. 
From the definition of $\Phi'$ immediately following \eqref{E:Phi s'}, we calculate that
\begin{equation}\label{Phi'inv}
\Phi\pinv = 
\left (\begin{matrix}
\frac{1}{(t-\theta)^{d_1}}  & 0 & 0 & \hdots & 0\\
\frac{-Q_1\twistinv}{(t-\theta)^{d_2}} & \frac{1}{(t-\theta)^{d_2}} & 0 & \hdots & 0 \\
\frac{(Q_1Q_2)\twistinv}{(t-\theta)^{d_3}} & \frac{-Q_2\twistinv}{(t-\theta)^{d_3}} & \frac{1}{(t-\theta)^{d_3}} & \hdots & 0\\
\vdots & \vdots & \vdots & \ddots & \vdots \\
\frac{(-1)^{r-1} (Q_1\dots Q_{r-1})\twistinv}{(t-\theta)^{d_r}}&\frac{(-1)^{r-2}(Q_2\dots Q_{r-1})\twistinv}{(t-\theta)^{d_r}} &\frac{(-1)^{r-3}(Q_3\dots Q_{r-1})\twistinv}{(t-\theta)^{d_r}} & \hdots & \frac{1}{(t-\theta)^{d_r}}
\end{matrix}\right ) \in \Mat_r(\overline K(t)).
\end{equation}
So for each $n \geq 0$ and $1 \leq \ell \leq r$, the $\ell$-th component of
$\bw \twistk{n} \prod_{1 \leq m \leq n} (\Phi \pinv) \twistk{n+1-m}$
is
\[ 
\sum_{\substack{\ell \leq i \leq r \\ 1 \leq j \leq d_i}} (-1)^{i-\ell} c_{i,j}^{q^n} (t - \theta^{q^{n}})^{d_{i} - j} \sum_{\ell = k_0 \leq k_1 \leq \cdots \leq k_n = i} \ \prod_{1 \leq m \leq n} \dfrac{\prod_{k_{m-1} \leq k < k_{m}} Q_{k} \twistk{m-1}}{(t - \theta^{q^{m}})^{d_{k_{m}}}}. 
\]

Then we calculate the Gauss norm for the general term:
\[
 \left \lVert (t - \theta^{q^{n}})^{d_{i} - j} \prod_{1 \leq m \leq n} \dfrac{\prod_{k_{m-1} \leq k < k_{m}} Q_{k} \twistk{m-1}}{(t - \theta^{q^{m}})^{d_{k_{m}}}} \right \rVert_{\theta} = q^{\alpha} 
\]
where $\alpha$ satisfies
\begin{align*}
\alpha & \leq (d_{i} - j) q^{n} + \sum_{1 \leq m \leq n} \left( \sum_{k_{m-1} \leq k < k_{m}} \left( \dfrac{s_{k} q^{m}}{q-1} + \mu_{k} \right) - d_{k_{m}} q^{m} \right) \\
& = \sum_{k_{0} \leq k < k_{n}} \mu_{k} + (d_{i} - j) q^{n} + \sum_{1 \leq m \leq n} \left( \dfrac{d_{k_{m-1}} - d_{k_{m}}}{q-1} q^{m} - d_{k_{m}} q^{m} \right) \\
& = \sum_{\ell \leq k < i} \mu_{k} + (d_{i} - j) q^{n} + \dfrac{d_{k_0} q}{q-1} - \dfrac{d_{k_{n}} q^{n+1}}{q-1} + \sum_{1 \leq m \leq n} d_{k_{m}} \left( \dfrac{q^{m+1}-q^{m}}{q-1} - q^{m} \right) \\
& = \sum_{\ell \leq k < i} \mu_{k} + \dfrac{d_{\ell} q}{q-1} - \left( j+\dfrac{d_i}{q-1} \right) q^n.
\end{align*} 
Therefore we have
\begin{align*}
\left \lVert \bw \twistk{n}\prod_{1 \leq m \leq n} (\Phi \pinv) \twistk{n+1-m} \right \rVert_{\theta}
 \leq &\max_{\substack{1 \leq \ell \leq i \leq r \\ 1 \leq j \leq d_i}} \left\{ q^{(\mu_{\ell} + \cdots + \mu_{i-1}) + \frac{d_{\ell} q}{q-1}} \left( \dfrac{  |c_{i,j}|_{\infty} }{q^{j+\frac{d_i}{q-1}}} \right)^{q^n} \right\},
\end{align*}
which goes to 0 as $(n \to \infty)$.  Since $\TT_{\theta}$ is complete with respect to $\lVert \cdot \rVert_{\theta}$, we have the desired result.

\end{proof}
\begin{remark}
Note that the condition on $|c_{i,j}|_{\infty}$ coincides with \cite[Prop.~4.2.2]{CM17b}.
\end{remark}

\begin{remark}\label{g:difference-eq} 
 Note that from the definition one sees that $\bg$ satisfies the difference equation
\[ \bg^{(-1)}\Phi'-\bg=\bw  .\]

\end{remark}

\begin{remark}\label{R: DeltaIota_a}
Using Proposition~\ref{P: Taylor coeff}, we can rewrite $\delta_{0}\circ \iota (\ba)$ in Proposition~\ref{P: delta0} as
\begin{align} \label{E: DeltaIota_a} 
\begin{array}{r}
\begin{array}{rrccccccccc}
& & \multicolumn{4}{c}{\overbrace{\hspace{14.0em}}^{d_{1}}} & \multicolumn{4}{c}{\overbrace{\hspace{14.0em}}^{d_{2}}} & \\
\delta_0 \circ \iota(\ba) = \hspace{-0.5em} & \ldelim({1}{0em} & (\partial_{t}^{d_{1}-1} a_{1})|_{t=\theta}, \hspace{-0.6em} & \ldots, \hspace{-0.6em} & (\partial_{t}^{1} a_{1})|_{t=\theta}, \hspace{-0.6em} & a_{1}(\theta), \hspace{-0.6em} & (\partial_{t}^{d_{2}-1} a_{2})|_{t=\theta}, \hspace{-0.6em} & \ldots, \hspace{-0.6em} & (\partial_{t}^{1} a_{2})|_{t=\theta}, \hspace{-0.6em} & a_{2}(\theta), \hspace{-0.6em} & \ldots \hspace{1.0em}
\end{array} \\
\begin{array}{cccccl}
& \multicolumn{4}{c}{\overbrace{\hspace{14.0em}}^{d_{r}}} & \\
\ldots, \hspace{-0.6em} & (\partial_{t}^{d_{r}-1} a_{r})|_{t=\theta}, \hspace{-0.6em} & \ldots, \hspace{-0.6em} & (\partial_{t}^{1} a_{r})|_{t=\theta}, \hspace{-0.6em} & a_{r}(\theta) & \hspace{-0.6em} \rdelim){1}{0.5em}^{\tr}.
\end{array}
\end{array} 
\end{align}
Moreover, Theorem~\ref{T: Anderson} can be reformulated via hyperderivatives as the following. Let the notation be as given  in Theorem~\ref{T: Anderson} and put $\ba:=\bg+\bw$. Then we have
$\Exp_G \left(\delta_0 \circ \iota(\ba) \right)=\delta_{1}\circ \iota(\bw)$, where $\delta_0 \circ \iota(\ba)$ is given in (\ref{E: DeltaIota_a}).

\end{remark}

\subsection{The formulae} We continue the notation given at the beginning of this section. Define
\[
\Theta_{i, j} := (-1)^{j-i} \dfrac{ \prod_{i \leq k < j} Q_{k}\twistk{-1}}{(t - \theta)^{d_j}} \ (1 \leq i \leq j \leq r)
\textrm{ and }
\Theta := \left( \begin{matrix} \Theta_{1, 1} & \cdots & \Theta_{1, r} \\ & \ddots & \vdots \\ & & \Theta_{r, r} \end{matrix} \right)\]
and note that $\Theta = (\Phi\pinv)^\tr \in \Mat_{ r }(\oK(t))$. Let $X\in \Mat_{r\times d}(A[t])$ be defined in \eqref{E: Def of X} using 
\[
\bx_j :=  \left( (t - \theta)^{d_j-1}, (t - \theta)^{d_j-2}, \dots, 1 \right) \hbox{ for }1\leq j\leq r. 
\]
Let $D := \Mat_{r \times d} (\oK(t)) \to \Mat_{ d} (\oK(t))$  be the operator given in \eqref{E: Def of D} The main result of this subsection is the following technical lemma, which was the main ingredient in the proof of Theorem \ref{T:Logcoords}. The essence of this technical lemma is to give a formula for the logarithm of the $t$-module $G$, evaluated at a sufficiently small argument, in terms of $\Phi$, the defining matrix of the dual $t$-motive $M$ associated to $G$.
\begin{lemma}[Technical Lemma]\label{L:technical lemma}

 Assume that $\lVert Q_i \rVert_{1} \leq q^{\frac{s_i q}{q-1}}$ for each $1 \leq i \leq r-1 $. Let $c_{i,j} \in \CC_{\infty}$ satisfy the condition that $|c_{i,j}|_{\infty} < q^{j+\frac{d_i}{q-1}}$ for each $1 \leq i \leq r$ and $1 \leq j \leq d_i$.  Let $G$ be the $t$-module defined in \S \ref{Subsec: t-module G}. Then the following identity holds:
\begin{equation} \label{logformula}
\Exp_{G}\left( \sum_{n = 0}^{\infty} \left. D \left( \left( \prod_{1 \leq m \leq n} \Theta \twistk{m} \right) X \twistk{n} \right) \right|_{t = \theta} \left( \begin{matrix} c_{1, 1} \\ \vdots \\ c_{r, d_r} \end{matrix} \right)^{\!\!\! (n)}\right)= \left( \begin{matrix} c_{1,1} \\ \vdots \\ c_{r,d_r} \end{matrix} \right).
\end{equation}
\end{lemma}
\begin{proof}
 Firstly, we remark that $\delta_0 \circ \iota : \Mat_{1\times r}(\TT_{\theta}) \to \CC_{\infty}^{d}$ is continuous. Indeed, let $\bff = (f_{1}, \dots, f_{r}) \in \Mat_{1 \times r}(\TT_{\theta})$ with $f_{i} = \sum_{j=0}^{\infty} b_{i, j} t^{j}$. Since by Proposition~\ref{P: delta0} the $(d_{1} + \cdots + d_{i-1} + j)$-th coordinate of $\delta_0 \circ \iota(\bff)$ is
\[ \sum_{\ell \geq d_{i}-j} {{\ell}\choose{d_{i}-j}} b_{i, \ell} \theta^{\ell-d_{i}+j} \]
for each $1 \leq i \leq r$ and $1 \leq j \leq d_{i}$, we have
\[  \lVert \delta_0 \circ \iota(\bff) \rVert  \leq \max_{\substack{1 \leq i \leq r \\ 1 \leq j \leq d_{i} \\ \ell \geq d_{i}-j}} \left\{ \left| b_{i, \ell} \theta^{\ell-d_{i}+j} \right|_{\infty} \right\} \leq \max_{\substack{1 \leq i \leq r \\ 1 \leq j \leq d_{i} \\ \ell \geq d_{i}-j}} \left\{ \left| b_{i, \ell} \theta^{\ell} \right|_{\infty} \right\} = \max_{\substack{1 \leq i \leq r \\ \ell \geq 0}} \left\{ \left| b_{i, \ell} \theta^{\ell} \right|_{\infty} \right\} = \lVert \bff \rVert_{\theta}. \]
Therefore, $\delta_0 \circ \iota$ is continuous.

Now we consider
\[
\delta_0 \circ \iota \left( \bw \twistk{n} \prod_{1 \leq m \leq n} (\Phi \pinv) \twistk{n+1-m} \right)
\]
with
\[
\bw := (w_1, \dots, w_r), \ \ \ w_j := \sum_{\ell = 1}^{d_j} c_{j, \ell} (t - \theta)^{d_j - \ell}
\]
for each $n \geq 0$.
Putting
\[
(\phi^{<n>}_{i, j}) :=  \left(\prod_{1 \leq m \leq n} \Theta \twistk{m} \right)^{\tr}
\]
and
\begin{equation}\label{cin}
(c_1^{<n>}(t), \dots, c_r^{<n>}(t))
:= \bw \twistk{n} \prod_{1 \leq m \leq n} (\Phi \pinv) \twistk{n+1-m} = \bw \twistk{n} \left( \prod_{1 \leq m \leq n} \Theta \twistk{m} \right)^{\tr},
\end{equation}
then we have that
\[
c_i^{<n>}(t) = \sum_{j = i}^{r} \sum_{\ell = 1}^{d_j} c_{j, \ell}^{q^n} (t - \theta^{q^n})^{d_j - \ell} \phi^{<n>}_{j, i} \\
= \sum_{j = i}^{r} \phi^{<n>}_{j,i} \bx_{j} \twistk{n} \left( \begin{matrix} c_{j, 1} \\ \vdots \\ c_{j, d_j} \end{matrix} \right)^{\!\!\! (n)}
\]
and hence
\begin{align}\label{E: din}
\bd_i^{<n>} &:= \left. \left( \begin{matrix} \partial_t^{d_i-1} c_i^{<n>} \\ \vdots \\  \partial_t^{1} c_i^{<n>} \\ c_i^{<n>} \end{matrix} \right) \right|_{t=\theta}
= {\displaystyle \sum_{j=i}^{r}} \left. \partial_t^{d_i} \left[ \phi^{<n>}_{j,i} \bx_{j} \twistk{n} \right] \right|_{t = \theta} \left( \begin{matrix} c_{j, 1} \\ \vdots \\ c_{j, d_j} \end{matrix} \right)^{\!\!\! (n)} \\
&= {\displaystyle \sum_{j=i}^{r}} \left. \Upsilon_{i, j}^{<n>} \right|_{t = \theta} \left( \begin{matrix} c_{j, 1} \\ \vdots \\ c_{j, d_j} \end{matrix} \right)^{\!\!\! (n)}
= \left. \left( 0, \dots, 0, \Upsilon_{i, i}^{<n>}, \Upsilon_{i, i+1}^{<n>}, \dots, \Upsilon_{i, r}^{<n>} \right) \right|_{t = \theta} \left( \begin{matrix} c_{1, 1} \\ \vdots \\ c_{r, d_r} \end{matrix} \right)^{\!\!\! (n)}, \nonumber
\end{align}
where
\begin{align*}
\Upsilon_{i, j}^{<n>}
:= & \ \partial_t^{d_i} \left[ \phi^{<n>}_{j,i} \bx_{j} \twistk{n} \right].
\end{align*}
Note that the second equality of (\ref{E: din}) comes from the definition of $\partial$-matrices~(\ref{E: partial matrix}), and the last two equalities are from definitions.
Based on the definition~(\ref{cin}), by Remark~\ref{R: DeltaIota_a} and the above calculations we have that
\begin{align*}
\delta_0 \circ \iota \left( \bw \twistk{n} \prod_{1 \leq m \leq n} (\Phi \pinv) \twistk{n+1-m} \right)
&= \left( \begin{matrix} \bd_1^{<n>} \\ \vdots \\ \bd_r^{<n>} \end{matrix} \right)
= \left. \left( \begin{matrix} \Upsilon_{1, 1}^{<n>} & \cdots & \Upsilon_{1, r}^{<n>} \\ & \ddots & \vdots \\ & & \Upsilon_{r, r}^{<n>} \end{matrix} \right) \right|_{t = \theta} \left( \begin{matrix} c_{1, 1} \\ \vdots \\ c_{r, d_r} \end{matrix} \right)^{\!\!\! (n)} \\
&= \left. D \left( \left(\prod_{1 \leq m \leq n} \Theta \twistk{m} \right) X \twistk{n} \right) \right|_{t = \theta} \left( \begin{matrix} c_{1, 1} \\ \vdots \\ c_{r, d_r} \end{matrix} \right)^{\!\!\! (n)}.
\end{align*}
It  follows by Theorem~\ref{T: Anderson}, Proposition~\ref{P:Convergence} and Remark~\ref{g:difference-eq}  that 
\begin{equation*}
 \Exp_{G}  \left( \sum_{n = 0}^{\infty} \left. D \left( \left(\prod_{1 \leq m \leq n} \Theta \twistk{m} \right) X \twistk{n} \right) \right|_{t = \theta} \left( \begin{matrix} c_{1, 1} \\ \vdots \\ c_{r, d_r} \end{matrix} \right)^{\!\!\! (n)} \right)=\left( \begin{matrix} c_{1, 1} \\ \vdots \\ c_{r, d_r} \end{matrix} \right),
\end{equation*} since from the definition~(\ref{E: Def of delta1}) of $\delta_{1}$ for the given $\bw$ one has 

\[ \delta_{1}\circ \iota (\bw)=
( c_{1, 1}, \ldots, c_{r, d_r} )^{\tr}.
\] 

\end{proof}

\begin{remark}

The shtuka function is defined for rank 1 Drinfeld modules and caries a great deal of arithmetic information (see \cite[\S 7.7-8.2]{T04}).  We view the matrix $\Phi'$ as a matrix analogue of the shtuka function and this motivates some of the constructions in this paper.  Indeed, the shtuka function for the Carlitz module is defined as $f=t-\theta \in A[t]$ and  if one sets $\fs=s_1=1$ ($r=1$), then $\Phi' = f(t)$.
Further, one can express the coefficients of the logarithm function associated to a rank 1 Drinfeld module in terms of the reciprocal of the shtuka function.  For the Carlitz module one has (due to Anderson, see \cite[Prop. 0.3.8]{T93})
\begin{equation}\label{Carlitzlog}
\log_C(z) = \sum_{i\geq 0} \frac{z^{q^i}}{f\twist \dots f\twisti}\Big|_{t=\theta}.
\end{equation}
Recalling that $\Theta = (\Phi\pinv)^\tr $, there is a natural comparison between our formula \eqref{logformula} and formula \eqref{Carlitzlog}.  There are several other arithmetic applications of the shtuka function (see \cite{GP16}, \cite{G17b} and \cite{ANT17}), and it would be interesting to study how they apply in this setting.

\end{remark}

\section{Anderson-Thakur series}\label{Section of t-motivic MZV's}
In this section we study Anderson-Thakur series, which are deformations of the MZV's in \eqref{MZVdef}.  We then establish explicit formulae for Anderson-Thakur series in terms of $t$-motivic CMSPL's.

\subsection{Definition of Anderson-Thakur series}
\subsubsection{Anderson-Thakur polynomials}\label{AT polynomials}
For any non-negative integer $n$, we let $H_{n}(t)\in A[t]$ be the Anderson-Thakur polynomial defined in~\cite{AT90, AT09}, but we follow the notation given in~\cite{C14, CPY14}.  Namely, we first let $x,y$ be two independent variables and put $G_{0}(y):=1$ and define the polynomials $G_{n}(y)\in \FF_{q}[t,y]$ for positive integers $n$ by 
\[G_{n}(y):=\prod_{i=1}^{n}(t^{q^{n}}-y^{q^{i}}). \]

Put $D_{0}:=1$ and $D_{i}:=\prod_{j=0}^{i-1}(\theta^{q^{i}}-\theta^{q^{j}})\in A$ for $i\in \N$. For any non-negative integer $n$, recall that the Carlitz factorial is defined by
\begin{equation}\label{E: Carlitz factorials}
\Gamma_{n+1}:=\prod_{i=0}^{\infty}D_{i}^{n_i}\in A 
\end{equation}
where the $n_i \in \Z_{\geq 0}$ are given by writing the base $q$-expansion $n=\sum_{i=0}^{\infty} n_{i}q^{i}$ for $0\leq n_i\leq q-1$. We then define the sequence of Anderson-Thakur polynomials~\cite{AT90} $H_{n}(t)\in A[t]$ by the following generating function identity,
\[ \left( 1-\sum_{i=0}^{\infty} \frac{G_{i}(\theta) }{ D_{i}|_{\theta=t} }  x^{q^{i}}\right)^{-1} =\sum_{n=0}^{\infty}  \frac{H_{n}(t) }{ \Gamma_{n+1}|_{\theta=t}}  x^{n}  .\]
For any  non-negative integer $d$, we denote by $A_{d,+}$ the set of elements of degree $d$ in $A_{+}$.  Anderson and Thakur showed in \cite{AT90} that for any positive integer $s$,
\begin{equation}\label{E:BoundATpolynomials}
 \| H_{s-1}(t)\|_{1} <|\theta|_{\infty}^{\frac{sq}{q-1}},   
\end{equation}
and that the interpolation formula holds for every $d \in  \ZZ_{\geq 0}$:
\begin{equation}\label{E:InterpolationAT}
\dfrac{H_{s-1}^{(d)}(\theta)}{L_{d}^{s}} =\Gamma_{s} \cdot \sum_{a\in A_{d,+}}\frac{1}{a^{s}}  .
\end{equation}

\subsubsection{The definition of Anderson-Thakur series}
 In \cite[\S 2.5]{AT09}, Anderson and Thakur construct deformation series for MZV's.  We use a similar construction here, but it differs from Anderson and Thakur's construction by a factor of powers of $\Omega$, and thus has slightly different evaluation properiets.  Due to the similarities, we still call it an \textit{Anderson-Thakur series}.  Recall the notation $\LL_{i}$ given in~\eqref{E: LLi}.

\begin{definition}\label{Def: motivicMZV} For any index $\fs=(s_{1},\ldots,s_{r})\in \NN^{r}$, we define the Anderson-Thakur series associated to $\fs$ by the following series 
\[ 
\zetaAmot(\fs):= \sum_{i_{1}>\cdots>i_{r}\geq 0} \dfrac{H_{s_{1}-1}^{(i_{1})}\cdots H_{s_{r}-1}^{(i_{r})} } {\LL_{i_{1}}^{s_{1}}\cdots \LL_{i_{r}}^{s_{r}} }  \in  \TT_{\theta}.
 \]
\end{definition}

Note that $\zetaAmot(\fs)$ may not be entire. For example, let $q=3$ and $\fs=(2)$. Then we have $H_1=1$ and so from the defining series, $\zetaAmot(2)$ has poles at $t=\theta^{q^{i}}$ for each $i\in \NN$.

We now fix a fundamental period $\tilde{\pi}$ of the Carlitz $\FF_{q}[t]$-module $\bC$, i.e., $\Ker \Exp_{\bC}=A \cdot \tilde{\pi}$. Put
\begin{equation}\label{D:OmegaFunc}
\Omega(t):=(-\theta)^{\frac{-q}{q-1}} \prod_{i=1}^{\infty} \biggl(
1-\frac{t}{\theta^{q^{i}}} \biggr)\in \power{\CC_{\infty}}{t},
\end{equation}

where $(-\theta)^{\frac{1}{q-1}}$ is a suitable choice of  $(q-1)$-st root of $-\theta$ satisfying $\frac{1}{\Omega(\theta)}=\tilde{\pi}$ (see~\cite{ABP04, AT09}). We note that $\Omega$ satisfies the functional equation
\[ \Omega^{(1)}=\Omega/(t-\theta^{q}), \]
and hence
\[ \Omega^{s_{1}+\cdots+s_{r}}\cdot \zetaAmot(\fs)=\sum_{i_{1}>\cdots>i_{r}\geq 0}(\Omega^{s_{r}}H_{s_{r}-1})^{(i_{r})}\cdots (\Omega^{s_{1}} H_{s_{1}-1})^{(i_{1})} ,  \]
which is an entire series studied in \cite{AT09}. 

\begin{remark}
 By the Anderson-Thakur's interpolation formula~(\ref{E:InterpolationAT}) we have that 
\begin{equation}\label{MMZV=MZV}
\zetaAmot(\fs)|_{t=\theta}=\Gamma_{\fs}\zeta_{A}(\fs),
\end{equation}
where $\Gamma_{\fs}:=\Gamma_{s_{1}}\cdots \Gamma_{s_{r}}$. Note that $\Omega^{s_{1}+\cdots+s_{r}} \zetaAmot(\fs)$ is an entire power series (see~\cite{AT09, CPY14}). Since $\Omega$ is entire on $\CC_{\infty}$ with simple zeros at $t=\theta^{q},\theta^{q^{2}},\cdots$, the series $\zetaAmot(\fs)$ lies in $\TT_{\theta}$ and hence $\partial_{t}^{j} \zetaAmot(\fs)\in \TT_{\theta}$ for every positive integer $j$. 
\end{remark}

\subsection{Explicit formulae for Anderson-Thakur series} For any index  $\fs=(s_{1},\ldots,s_{r})\in \NN^{r}$, combining Anderson-Thakur's work on the interpolation formula with \cite[Thm.~5.2.5]{CM17b} we can express $\zeta_{A}(\fs)$ as
\begin{equation}\label{E: FormulaMZV}
 \Gamma_{\fs}\zeta_{A}(\fs)=\sum_{\ell=1}^{T_\fs} b_{\ell}(\theta)\cdot (-1)^{\dep(\fs_{\ell})-1}\Lis_{\fs_{\ell}}(\bu_{\ell}), 
\end{equation}
for some number $T_\fs \in \NN$, explicit coefficients $b_{\ell}(t)\in \FF_{q}[t]$, explicit indexes  $\fs_{\ell}\in \NN^{\dep(\fs_{\ell})}$ with $\dep(\fs_{\ell})\leq \dep(\fs)$ and $\wt(\fs_{\ell})=\wt(\fs)$ and explicit integral points $\bu_{\ell}\in A^{\dep(\fs_{\ell})}$.

The aim of this section is to deform the identity (\ref{E: FormulaMZV}) to an identity of power series.
\begin{lemma}\label{L:Deformation identity}
Fix an index $\fs=(s_{1},\ldots,s_{r})\in \NN^{r}$. Let $(\fs_{\ell},\bu_{\ell},b_{\ell}(t))$ and $T_\fs$ be given in (\ref{E: FormulaMZV}). Then we have the following identity
\begin{equation}\label{E:FormulaMotivicMZV}
\zetaAmot(\fs)=\sum_{\ell=1}^{T_\fs} b_{\ell}(t)\cdot (-1)^{\dep(\fs_{\ell})-1}\fLis_{\fs_{\ell},\bu_{\ell}}  .
\end{equation}

\end{lemma}

\begin{remark} We mention in this remark that each value $\Lis_{\fs_{\ell}}(\bu_{\ell})$ occurring in (\ref{E: FormulaMZV}) is non-vanishing  (when $\bu_{\ell} \in (A \setminus \{ 0 \})^{\dep(\fs_{\ell})}$). We simply argue as follows. For any index $\fs=(s_{1},\ldots,s_{r})\in \NN^{r}$, we put
\[ 
\mathbb{D}_{\fs}':=\left\{(z_{1},\ldots,z_{r})\in \CC_{\infty}^{r}; |z_{i}|_{\infty}< q^{\frac{s_{i}q}{q-1}} \hbox{ for }i=1,\ldots,r \right\}
\]
 and
\[
\mathbb{D}_{\fs}'':=\left\{(z_{1},\ldots,z_{r})\in \CC_{\infty}^{r}: |z_{1}|_{\infty} < q^{\frac{s_{1}q}{q-1}} \hbox{ and } |z_{i}|_{\infty} \leq q^{\frac{s_{i}q}{q-1}} \hbox{ for }i=2,\ldots,r \right\}.
\] Then for any $\bu=(u_{1},\ldots,u_{r})\in \mathbb{D}_{\fs}''$, by \cite[Rem.~5.1.4]{C14} the absolute value of the general terms of $\Lis_{\fs}(\bu)$ is given by
\[ \left |\frac{ u_1^{q^{i_1}}  \dots u_r^{q^{i_r}}}{ L_{i_1}^{s_1}  \dots L_{i_r}^{s_r}}   \right |_{\infty} = q^{\frac{q}{q-1}(s_1+\cdots+s_r)}\cdot 
\left |\frac{u_1}{\theta^{s_1 q/(q-1)}}\right |_{\infty}^{q^{i_1}} \cdots \left |\frac{u_r}{\theta^{s_r q/(q-1)}}\right |_{\infty}^{q^{i_r}}.\]
Therefore the absolute values above have a unique maximal one when $i_1=\cdots=i_r=0$ for any $\bu \in \mathbb{D}_{\fs}'' \cap (\CC_{\infty}^{\times })^{r}$.
Note that $\bu_{\ell} \in \mathbb{D}_{\fs_{\ell}}' \subset \mathbb{D}_{\fs_{\ell}}''$ by (\ref{E:BoundATpolynomials}) and \cite[Prop.~5.2.2 and Rmk.~5.2.6]{CM17b}, whence $\Lis_{\fs_{\ell}}(\bu_{\ell})\neq 0$. 
\end{remark}

\subsubsection{Formula of $\zetaAmot(\fs)$ in terms of $\fLi_{\fs,\bu}$}  We will give the proof of Lemma \ref{L:Deformation identity} shortly, but first we discuss some new ideas in order to connect $\zetaAmot(\fs)$ to $\fLi_{\fs,\bu}$.  Fix an index $\fs=(s_{1},\ldots,s_{r})\in \NN^{r}$. For each $1\leq i \leq r$, we expand the Anderson-Thakur polynomial $H_{s_{i}-1}(t)$ as 
\[H_{s_{i}-1}(t)=\sum_{j=0}^{n_{i}} u_{ij}t^{j}, \]
where $u_{ij}\in A$ with $|u_{ij}|_{\infty}<|\theta|_{\infty}^{\frac{s_{i}q}{q-1}} $ and $u_{in_{i}}\neq 0$. Following the notation of \cite{CM17b} we define
\[ J_{\fs}:=\left\{0,1,\ldots,n_{1} \right\}\times \cdots \times \left\{0,1,\ldots,n_{r} \right\}  .\]
For each $\bj=(j_{1},\ldots,j_{r})\in J_{\fs}$, we put 
\begin{equation}\label{budef}
\bu_{\bj}:=(u_{1 j_{1}},\ldots,u_{r j_{r}} )\in A^{r} \ {\textnormal{and}} \ a_{\bj}(t):=t^{j_{1}+\cdots+j_{r}}\in \FF_{q}[t].
\end{equation}
One then observes that
\[H_{s_{1}-1}^{(i_{1})}\cdots H_{s_{r}-1}^{(i_{r})}= \left( \sum_{j=0}^{n_{1}} u_{1 j}^{q^{i_{1}}}t^{j} \right) \cdots \left(\sum_{j=0}^{n_{r}} u_{r j}^{q^{i_{r}}}t^{j} \right)
=\sum_{\bj=(j_{1},\ldots,j_{r})\in J_{\fs}} a_{\bj}(t)u_{1j_{1}}^{q^{i_{1}}}\cdots u_{r j_{r}}^{q^{i_{r}}}.  \]
Dividing the above equality by $\LL_{i_{1}}^{s_{1}}\cdots \LL_{i_{r}}^{s_{r}}$ and then summing over all $i_{1}>\cdots i_{r}\geq 0$ we find the following identity from the definitions of $\zetaAmot(\fs)$ and $\fLi_{\fs,\bu}$.
\begin{proposition}\label{P:zetaCMPL}
Let $\fs=(s_{1},\ldots,s_{r})\in \NN^{r}$ and let $J_{\fs}$ be defined as above. Then we have the following identity
\begin{equation}\label{E:MMZV=MCMPL}
\zetaAmot(\fs)=\sum_{\bj\in J_{\fs}}a_{\bj}(t)\fLi_{\fs,\bu_{\bj}}.
\end{equation}
\end{proposition}
\begin{remark}
When we specialize the both sides of (\ref{E:MMZV=MCMPL}) at $t=\theta$, then we obtain the identity 
\begin{equation}\label{E:MZV=CMPL}
\Gamma_{\fs}\zeta_{A}(\fs)=\sum_{\bj\in J_{\fs}}a_{\bj}(\theta)\Li_{\fs}(\bu_{\bj})
\end{equation}
 given in \cite[Thm.~5.5.2]{C14}. 
\end{remark}

\subsubsection{Review of the identity~\textnormal{(\ref{E: FormulaMZV})}}\label{Sub:ReviewFormulaMZV}

We first mention that the arguments of proving the identity (\ref{E:FormulaMotivicMZV}) are essentially the same as the arguments of deriving (\ref{E: FormulaMZV}), and so we quickly review the ideas how we derive (\ref{E: FormulaMZV}). As we have the formula~(\ref{E:MZV=CMPL}), it suffices to express the CMPL $\Li_{\fs}(\bu_{\bj})$ of the right hand side of (\ref{E:MZV=CMPL}) in terms of linear combination of $\Lis_{\fs_{\ell}}$ with coefficients $\pm 1$. It is easy to achieve this goal by using inclusion-exclusion principle on the set
\[ \left\{ i_{1}>\cdots>i_{r}\geq 0 \right\}. \]
Note that the coefficients $b_{\ell}(\theta)$ arise from some $a_{\bj}(\theta)$ up to $\pm 1$. 

\subsubsection{Proof of Lemma~\ref{L:Deformation identity}} Now we prove the identity (\ref{E:FormulaMotivicMZV}). First, we start with the identity (\ref{E:MMZV=MCMPL}).  We then use inclusion-exclusion principle on the set 
\[ \left\{ i_{1}>\cdots>i_{r}\geq 0 \right\}\]
to express the $\fLi_{\fs,\bu_{\bj}}$ of the right hand side of (\ref{E:MMZV=MCMPL}) as a linear combinations of  $\fLis_{\fs_{\ell},\bu_{\ell}}$ with coefficients $\pm 1$. Since such a procedure is completely the same as in Sec.~\ref{Sub:ReviewFormulaMZV} and since the coefficients of the right hand side of (\ref{E:MMZV=MCMPL}) and (\ref{E:MZV=CMPL}) are the same when replacing $t$ by $\theta$, by going through the details we obtain the desired formula (\ref{E:FormulaMotivicMZV}). 

\begin{remark}
The formula~(\ref{E:FormulaMotivicMZV}) will be used in the proof of Theorem~\ref{T:MainThm}. The key idea of the proof is that formula (\ref{E:FormulaMotivicMZV}) is the deformation of (\ref{E: FormulaMZV}).  In the proof of Theorem~\ref{T:MainThm} we do not need to know the precise coefficients $b_{\ell}$, so we avoid presenting the repetitive details given in \cite[Sec.~5.2]{CM17b}. One certainly could write down the precise coefficients $b_{\ell}$ by going through the procedure mentioned above.
\end{remark}

\section{Explicit formulae for $Z_{\fs}$}\label{S:ExplicitFormulae}
For a given index $\fs=(s_{1},\ldots,s_{r})\in\NN^{r}$, in~\cite{CM17b} the first and third authors explicitly construct a uniformizable $t$-module $G_{\fs}$ defined over $A$, a special point $\bv_{\fs}\in G_{\fs}(A)$ and a vector $Z_{\fs}\in\Lie G_{\fs}(\CC_{\infty})$ so that 
\begin{itemize}
\item $\Exp_{G_{\fs}}(Z_{\fs})=\bv_{\fs}$, and
\item the $d_{1}$-th coordinate of $Z_{\fs}$ gives $\Gamma_{\fs}\zeta_{A}(\fs)$, 
\end{itemize}
where we recall that $d_{i}:=s_{i}+\cdots+s_{r}$ for $i=1,\ldots,r$. The purpose of this section is to give explicit formulae for all coordinates of $Z_{\fs}$ in terms of hyperderivatives of Anderson-Thakur series and $t$-motivic CMSPL's. This is done in Theorem \ref{T:MainThm}.  We then also give examples of our main theorem, and apply it to monomials of MZV's.

\subsection{Review of the constructions of $G_{\fs},\bv_{\fs},Z_{\fs}$}\label{Sub: fiber coproduct M} In this subsection, for a fixed index $\fs=(s_{1},\ldots,s_{r})\in \NN^{r}$ we recall the constructions of the $t$-module $G_{\fs}$ defined over $A$, the special point $\bv_{\fs}\in G_{\fs}(A)$ and the vector $Z_{\fs}\in \Lie G_{\fs}(\CC_{\infty})$ given in~\cite{CM17b}. 
\subsubsection{Review of fiber coproducts of dual $t$-motives} Let $\KK\subset \CC_{\infty}$ be an algebraically closed subfield containing $K$. Fix a dual $t$-motive $\cN$, and let $\left\{ \cM_{1},\ldots,\cM_{T} \right\}$ be dual $t$-motives so that   $\cN \subseteq \cM_{i}$ is a left $\KK[t,\sigma]$-submodule   and the quotient $\cM_{i}/\cN$ is either zero or a dual $t$-motive for each $i$. Let $\bn$ be a $\KK[t]$-basis of $\cN$ and denote by $\bn_{i}$ the image of $\bn$ under the inclusion $\cN\hookrightarrow \cM_{i}$.  Under the assumptions on $\cM_{i}$, we note that for each $i$ the set $\bn_{i}$ is either a $\KK[t]$-basis of $\cM_{i}$ or can be extended to a $\KK[t]$-basis of $\cM_{i}$. 

We define $\cM$ to be the fiber coproduct of $\left\{ \cM_{i}\right\}_{i=1}^{T}$ over $\cN$ denoted by  $\cM_{1}\sqcup_{\cN}\cdots\sqcup_{\cN}\cM_{T} $. As a left $\KK[t]$-module, $\cM$ is defined by the following quotient module
\[ \cM:=\bigoplus_{i=1}^{T} \cM_{i}/ \left( \Span_{\KK[t]}\left\{ x_{i}'-x_{j}' | \ \forall \ x\in \bn \ \forall 1\leq i, j \leq T \right\} \right),\]
where $x_{i}'$ denotes the image of $x$ under the embedding $\cN\hookrightarrow \cM_{i}$ for $i=1,\ldots,T$. It is shown in \cite[Sec.~2.4.2]{CM17b} that the $\KK[t]$-module  $\left( \Span_{\KK[t]}\left\{ x_{i}'-x_{j}' | \ \forall \ x\in \bn \ \forall 1\leq i, j \leq T \right\} \right)$ is stable under the $\sigma$-action, and hence $\cM$ is a left $\KK[t,\sigma]$-module. In fact, $\cM$ is shown to be a dual $t$-motive.

\subsubsection{The set up} 
Recall the formula $\Gamma_{\fs}\zeta_{A}(\fs)=\sum_{\ell=1}^{T_\fs}  b_{\ell}(\theta)  \cdot (-1)^{\dep(\fs_{\ell})-1}\Lis_{\fs_{\ell}}(\bu_{\ell})$ given in (\ref{E: FormulaMZV}).  To simplify notation we fix $T = T_\fs$ and let $s$ be the cardinality of those triples $(\fs_{\ell},\bu_{\ell},b_{\ell}(\theta))$ with $\dep(\fs_{\ell})=1$.  We then renumber the indexes $\ell$ of $\left\{1,\ldots,T \right\}$ such that 
\begin{itemize}
\item $1\leq \ell \leq s$ if $\dep(\fs_{\ell})=1$, and
\item $s+1\leq \ell \leq T$ for $\dep(\fs_{\ell})\geq 2$,
\end{itemize}
where $\ell$ corresponds to the triple $(\fs_{\ell},\bu_{\ell},b_{\ell}(\theta))$.

Now for each $1\leq \ell \leq T$ we define matrices $\Phi_{\ell}\in \Mat_{\dep(\fs_{\ell})+1}(\KK[t])$ and $\Phi_{\ell}'\in \Mat_{\dep(\fs_{\ell})}(\KK[t])$ using the $\dep(\fs_{\ell})$-tuple  $\widetilde{\fs}_{\ell}$ (recall that $\widetilde \cdot$ reverses the order of the tuple) as in (\ref{E:Phi s}) and (\ref{E:Phi s'}), respectively, with $\mathfrak{Q} = \widetilde \bu_\ell$.  Further, define the Frobenius module $M_{\ell}$ and the dual $t$-motive $M_{\ell}'$ as in Sec.~\ref{Subsec: t-motives to t-modules} with sigma actions given by $\Phi_\ell$ and $\Phi_\ell'$, respectively. For each  $(\widetilde{\fs}_{\ell}, \widetilde{\bu}_{\ell})$, let $(G_{\ell},\rho_{\ell})$ be the $t$-module associated to $M_{\ell}'$, i.e., $\rho_{\ell}$ is given in (\ref{E:Explicit t-moduleCMPL}), and let $\bv_{\ell}$  be the special point of $G_{\ell}$ given in  (\ref{E:v_s,u}). Note that $\bv_{\ell}$ arises from $M_{\ell}$ by Remark~\ref{Rem: Ext1}.

Since each  $\widetilde{\bu}_{\ell}$ is an integral point, by Sec.~\ref{Sub:Gsu} we see that $G_{\ell}$ is defined over $A$ and $\bv_{\ell}\in G_{\ell}(A)$. Note further that by \cite[Thm.~4.2.3]{CM17b} the logarithm $\Log_{G_{\ell}}$ converges at the special point $\bv_{\ell}$, and the $d_{1}$-th coordinate of $\Log_{G_{\ell}}(\bv_{\ell})$ is equal to
\[ (-1)^{\dep(\fs_{\ell})-1} \Lis_{\fs_{\ell}} (\bu_{\ell}) ,\]
where $d_{1}:=\wt(\fs)=\wt(\fs_{\ell})$ for all $\ell$. Finally, we put $Z_{\ell}:=\Log_{G_{\ell}}(\bv_{\ell})$ for each $\ell$. We note that the above setting is the same as \cite[Sec.~5.3]{CM17b}.

\subsubsection{The $t$-module $G_{\fs}$} \label{Subsec: fiber coproduct}
 For each $\ell$, we let $\rho_{\ell}$ be the map defining the $t$-module structure on $G_{\ell}$. By (\ref{E:Explicit t-moduleCMPL}) we see that if $\ell\geq s+1$, $\rho_{\ell}(t)$ is  a right upper triangular block matrix with $[t]_{d_1}$ located as upper left square. So $\rho_{\ell}(t)$ has the shape
\[ 
\rho_\ell(t) = \begin{pmatrix}
[t]_{d_{1}}& F_{\ell}\\
& \rho_{\ell}(t)'
\end{pmatrix},
\]
where $[t]_{d_{1}}$ is defined in \eqref{[t]ndef} and where $F_\ell$ and $\rho_\ell'$ are matrices over $A[\tau]$ which one could calculate explicitly, but going forward we only need to know that $[t]_{d_1}$ is the top left block without knowing the precise sizes of $F_{\ell}$ and $\rho_{\ell}'$.  Note that if  $\ell \leq s$ then $\rho_{\ell}(t)=[t]_{d_{1}}$. We define the $t$-module $(G_{\fs},\rho)$ defined over $A$ to be the $t$-module associated to the dual $t$-motive
\[\cM_{\fs} =  M_{1}'\sqcup_{C^{\otimes d_{1}}}\cdots\sqcup_{C^{\otimes d_{1}}}M_{T}', \]
which is the fiber coproduct of the dual $t$-motives $\left\{ M_{\ell}'\right\}_{\ell=1}^{T}$ over $C^{\otimes d_{1}}$.  We claim that $\rho(t)$ is a block upper triangular matrix given by
\begin{equation}
\begin{pmatrix}\label{E: rho}
[t]_{d_{1}}&F_{s+1}&\cdots & F_{T}\\
& \rho_{s+1}(t)' & &\\
& & \ddots &\\
& & & \rho_{T}(t)'\\
\end{pmatrix},
\end{equation}
where again $F_i$ and $\rho_i'$ are certain matrices over $A[\tau]$ whose exact description we do not require going forward.  We will prove the above claim after giving a brief definition.
\begin{definition}\label{Def: z-}
For any vector $\bz\in \Mat_{n\times 1}(\CC_{\infty})$ with $n\geq d_{1}$, we define $\hat{\bz}$ as the vector of the first $d_{1}$ coordinates of $\bz$, and $\bz_{-}$ as the vector of the remaining $n-d_{1}$ coordinates, i.e.,
\[
\bz=\begin{pmatrix}
\hat{\bz}\\
{\bz}_{-}
\end{pmatrix}
\] for which $\hat{\bz}$ is of length $d_{1}$ and ${\bz}_{-}$ is of length $n - d_{1}$.
\end{definition}
To prove the claim, we note that from the construction of fiber coproducts of dual $t$-motives  there is a natural morphism of $t$-modules: 
\[ \pi:=\left(  \left(  \bz_{1}, \dots, \bz_{T} \right) \mapsto \left( \sum_{\ell=1}^{T} \hat{\bz}_{\ell}^{\tr},{\bz_{s+1}^{\tr}}_{-},\ldots, {\bz_{T}^{\tr}}_{-}\right)^{\tr}\right): \bigoplus_{\ell=1}^{T} G_{\ell} \rightarrow G_{\fs}. \]

 Note further that given a point $\bz \in G_{\fs}$, we can pick a suitable point $\bz_{\ell} \in G_{\ell}$ for each $1 \leq \ell \leq T$ so that $\hat{\bz}_{\ell} = \mathbf{0}$ for all $\ell \neq s+1$ and $\pi\left(  \bz_{1}, \dots, \bz_{T} \right) = \bz$.
Since $\pi$ is a morphism of $t$-modules, we have
\begin{align*}
\rho(t) (\bz)
&= \rho(t)\left( \pi\left( \bz_{1}, \dots, \bz_{T} \right)\right)
= \pi  \left(  \rho_{1}(t)(\bz_{1}), \dots,  \rho_{T}(t)(\bz_{T}) \right) \\
&= \pi \left( \mathbf{0}, \dots, \mathbf{0}, \begin{pmatrix} [t]_{d_{1}} \hat{\bz}_{s+1} + F_{s+1} {\bz_{s+1}}_{-} \\ \rho_{s+1}(t)' {\bz_{s+1}}_{-} \end{pmatrix}, \begin{pmatrix} F_{s+2} {\bz_{s+2}}_{-} \\ \rho_{s+2}(t)' {\bz_{s+2}}_{-} \end{pmatrix}, \ldots, \begin{pmatrix} F_{T} {\bz_{T}}_{-} \\ \rho_{T}(t)' {\bz_{T}}_{-} \end{pmatrix} \right) \\
&= \begin{pmatrix} [t]_{d_{1}} \hat{\bz}_{s+1} + {\displaystyle \sum_{\ell=s+1}^{T}} F_{\ell} {\bz_{\ell}}_{-} \\ \rho_{s+1}(t)' {\bz_{s+1}}_{-} \\ \vdots \\ \rho_{T}(t)' {\bz_{T}}_{-} \end{pmatrix}
= \begin{pmatrix}
[t]_{d_{1}}&F_{s+1}&\cdots & F_{T}\\
& \rho_{s+1}(t)' & &\\
& & \ddots &\\
& & & \rho_{T}(t)'\\
\end{pmatrix}
\begin{pmatrix} \hat{\bz}_{s+1} \\ {\bz_{s+1}}_{-} \\ \vdots \\ {\bz_{T}}_{-} \end{pmatrix} \\
&= \begin{pmatrix}
[t]_{d_{1}}&F_{s+1}&\cdots & F_{T}\\
& \rho_{s+1}(t)' & &\\
& & \ddots &\\
& & & \rho_{T}(t)'\\
\end{pmatrix}
\bz.
\end{align*}

Finally, we mention that the special point $\bv_{\fs}\in G_{\fs}(A)$  and the vector $Z_{\fs} \in \Lie G_{\fs}(\CC_{\infty})$ in \cite{CM17b} are defined by

\begin{equation}\label{E: Def of vs}
\bv_{\fs}:=\pi\left( (\rho_{1}(b_{1}(t))(\bv_{1}), \dots, \rho_{T}(b_{T}(t))(\bv_{T}))\right)
\end{equation}
and
\begin{equation}\label{E: Def of Zs}
Z_{\fs} := \partial \pi \left( (\partial \rho_{1}(b_{1}(t)) Z_{1},\dots,\partial \rho_{T}(b_{T}(t)) Z_{T}) \right), 
\end{equation}
and one has the fact \cite[Thm.~5.1.3]{CM17b} that
\[\Exp_{G_{\fs}}(Z_{\fs})=\bv_{\fs} .\]

\subsection{The main result}
 The primary result of this paper is stated as follows.

\begin{theorem}\label{T:MainThm} For any index $\fs=(s_{1},\ldots,s_{r})\in \NN^{r}$, we  let $d_{1}:=s_{1}+\cdots+s_{r}$ and let $Z_{\fs}\in \Lie G_{\fs}(\CC_{\infty})$ be the vector given as above. For each $s+1\leq \ell \leq T$, we  set
\[
\bY_{\ell}:=\bY_{\tilde{\fs}_{\ell},\tilde{\bu}_{\ell}},
\]which is defined in Theorem~\ref{T:Logcoords}. Then $Z_{\fs}$ has the following explicit formula 
\[
Z_{\fs}=\begin{pmatrix}
\left(\partial_{t}^{d_{1}-1}\zetaAmot(\fs)\right)|_{t=\theta} \\
\vdots\\
\left(\partial_{t}^{1}\zetaAmot(\fs)\right)|_{t=\theta}\\
\zetaAmot(\fs)|_{t=\theta}\\
\left(\partial \rho_{s+1}(b_{s+1})(\bY_{s+1})\right)_{-}\\
\vdots\\
\left( \partial \rho_{T}(b_{T}) (\bY_{T})\right)_{-}
\end{pmatrix}
=
\begin{pmatrix}
\left(\partial_{t}^{d_{1}-1}\zetaAmot(\fs)\right)|_{t=\theta} \\
\vdots\\
\left(\partial_{t}^{1}\zetaAmot(\fs)\right)|_{t=\theta}\\
\Gamma_{\fs}\zeta_{A}(\fs)\\
\left(\partial \rho_{s+1}(b_{s+1})(\bY_{s+1})\right)_{-}\\
\vdots\\
\left( \partial \rho_{T}(b_{T}) (\bY_{T})\right)_{-}
\end{pmatrix},\]
where $b_{\ell}\in \FF_{q}[t]$ is given in (\ref{E:FormulaMotivicMZV}) and $\partial \rho_{\ell}(b_{\ell})\left( \bY_{\ell}\right)$ is explicitly given in Corollary~\ref{C: partial bY} for each $s+1\leq \ell \leq T$, and the notation $(\cdot)_{-}$ is defined in Definition~\ref{Def: z-}.

\end{theorem}

\begin{proof}
 Note that by Proposition~\ref{P: formula of partial rho}, $\partial \rho_{\ell}( b_{\ell})$ is given explicitly as diagonal block matrices, and the first block matrix is given by
\begin{equation}\label{E:FirstBlock}
\begin{pmatrix}
b_{\ell}(\theta)& \left(\partial_{t}^{1}b_{\ell} \right)|_{t=\theta}&\cdots&\left(\partial_{t}^{d_{1}-1}b_{\ell} \right)|_{t=\theta} \\
&\ddots& \ddots& \vdots\\
&&\ddots&  \left(\partial_{t}^{1}b_{\ell} \right)|_{t=\theta} \\
&&& b_{\ell}(\theta)
\end{pmatrix}.
\end{equation}Recall by Theorem~\ref{T:Logcoords} that $Z_{\ell}=\begin{pmatrix}\hat{Z}_{\ell}\\{Z_{\ell}}_{-} \end{pmatrix}$, where 
\[ \hat{Z}_{\ell}=\begin{pmatrix}
(-1)^{\dep(\fs_{\ell})-1 }\left(\partial_{t}^{d_{1}-1} \fLis_{\fs_{\ell},\bu_{\ell}}\right)|_{t=\theta} \\
\vdots\\
(-1)^{\dep(\fs_{\ell})-1 }\left(\partial_{t}^{1} \fLis_{\fs_{\ell},\bu_{\ell}}\right)|_{t=\theta} \\
(-1)^{\dep(\fs_{\ell})-1 } \left(\fLis_{\fs_{\ell},\bu_{\ell}}\right)|_{t=\theta} 
\end{pmatrix}= \begin{pmatrix}
(-1)^{\dep(\fs_{\ell})-1 }\left(\partial_{t}^{d_{1}-1} \fLis_{\fs_{\ell},\bu_{\ell}}\right)|_{t=\theta} \\
\vdots\\
(-1)^{\dep(\fs_{\ell})-1 }\left(\partial_{t}^{1} \fLis_{\fs_{\ell},\bu_{\ell}}\right)|_{t=\theta} \\
(-1)^{\dep(\fs_{\ell})-1 } \fLis_{\fs_{\ell}}(\bu_{\ell})
\end{pmatrix}. \]

Since $Z_{\fs} := \partial \pi \left( \partial \rho_{1}(b_{1}(t)) Z_{1},\dots,\partial \rho_{T}(b_{T}(t)) Z_{T} \right)$, by the definition of $\pi$  we have
\begin{equation}\label{E: Zs in proof}
 Z_{\fs}= \left( \sum_{\ell=1}^{T} \widehat{\partial \rho_{\ell}(b_{\ell}) ({Z}_{\ell}})^{\tr}, \left({\partial\rho_{s+1}(b_{s+1})(Z_{s+1})^{\tr}}\right)_{-},\ldots, \left( {\partial\rho_{T}(b_{T})(Z_{T}) ^{\tr}}\right)_{-}\right)^{\tr} . 
 \end{equation}Recall that $\partial \rho_{\ell}(b_{\ell})$ is a diagonal block matrix with the first block given as (\ref{E:FirstBlock}), whence we have that
\[
\begin{array}{rl}
&\sum_{\ell=1}^{T} \widehat{\partial \rho_{\ell}(b_{\ell}) ({Z}_{\ell}})\\
&\\
=& \sum_{\ell=1}^{T} \begin{pmatrix}
b_{\ell}(\theta)& \left(\partial_{t}^{1}b_{\ell} \right)|_{t=\theta}&\cdots& \left(\partial_{t}^{d_{1}-1}b_{\ell} \right)|_{t=\theta} \\
&\ddots& \ddots& \vdots\\
&&\ddots&  \left(\partial_{t}^{1}b_{\ell} \right)|_{t=\theta} \\
&&& b_{\ell}(\theta)
\end{pmatrix} 
\begin{pmatrix}
(-1)^{\dep(\fs_{\ell})-1}\cdot \left( \partial_{t}^{d_{1}-1}\fLis_{\fs_{\ell},\bu_{\ell}}(t)\right)|_{t=\theta}\\
(-1)^{\dep(\fs_{\ell})-1}\cdot  \left( \partial_{t}^{d_{1}-2}\fLis_{\fs_{\ell},\bu_{\ell}}(t)\right)|_{t=\theta}\\
\vdots\\
(-1)^{\dep(\fs_{\ell})-1}\cdot \fLis_{\fs_{\ell},\bu_{\ell}}(t)|_{t=\theta}
\end{pmatrix}\\
& \\
=& 
\begin{pmatrix}
\sum_{\ell=1}^{T}\partial_{t}^{d_{1}-1} \left( (-1)^{\dep(\fs_{\ell})-1} \cdot b_{\ell}(t) \fLis_{\fs_{\ell},\bu_{\ell}}(t) \right)|_{t=\theta}\\
\sum_{\ell=1}^{T}\partial_{t}^{d_{1}-2} \left( (-1)^{\dep(\fs_{\ell})-1} \cdot b_{\ell}(t) \fLis_{\fs_{\ell},\bu_{\ell}}(t) \right)|_{t=\theta} \\
\vdots\\
\sum_{\ell=1}^{T} (-1)^{\dep(\fs_{\ell})-1} \cdot b_{\ell}(t) \fLis_{\fs_{\ell},\bu_{\ell}}(t)|_{t=\theta}
\end{pmatrix}\\
& \\
=& \begin{pmatrix}
\left(\partial_{t}^{d_{1}-1}\zetaAmot(\fs)\right)|_{t=\theta} \\
\vdots\\
\left(\partial_{t}^{1}\zetaAmot(\fs)\right)|_{t=\theta}\\
\zetaAmot(\fs)|_{t=\theta}
\end{pmatrix}= \begin{pmatrix}
\left(\partial_{t}^{d_{1}-1}\zetaAmot(\fs)\right)|_{t=\theta} \\
\vdots\\
\left(\partial_{t}^{1}\zetaAmot(\fs)\right)|_{t=\theta}\\
 \Gamma_{\fs}\zeta_{A}(\fs)
\end{pmatrix},
\end{array}
\]where the second equality comes from Proposition \ref{P:dmatfacts} (2) and the third equality arises from linearity of hyperderivatives and (\ref{E:FormulaMotivicMZV}).

Since by definition $G_{\ell}=G_{\tilde{\fs}_{\ell},\tilde{\bu}_{\ell}}$, $\bv_{\ell}=\bv_{\tilde{\fs}_{\ell},\tilde{\bu}_{\ell}}$ and $Z_{\ell}=\Log_{G_{\ell}}(\bv_{\ell})$,  we have $Z_{\ell}=\bY_{\ell}:=\bY_{\tilde{\fs}_{\ell},\tilde{\bu}_{\ell}}$ defined in Theorem~\ref{T:Logcoords} and so the explicit formulae of the remaining coordinates of $Z_{\fs}$ follow from (\ref{E: Zs in proof}) and Corollary~\ref{C: partial bY}.
\end{proof}

\begin{example}
Take $q = 2$ and $\fs = (1,3)$. The fourth coordinate of $Z_{(1,3)}$ is given in \cite[Example 5.4.2]{CM17b}; however we can give the other coordinates explicitly here. In this case, we have $\Gamma_{1} = 1$, $\Gamma_{3} = \theta^{2} + \theta$, $H_{1-1} = 1$, $H_{3-1} = t + \theta^{2}$, $J_{(1,3)} = \{ (0,0), (0,1) \}$, $\bu_{(0,0)} = (1,\theta^{2})$, $\bu_{(0,1)} = (1,1)$, $a_{(0,0)} = 1$, $a_{(0,1)} = t$.
Thus we have
\begin{eqnarray*}
(\theta^{2} + \theta) \zeta_{A}(1,3)
\!\!\! & = & \!\!\! \Li_{(1,3)}(1,\theta^{2}) + \theta \Li_{(1,3)}(1,1) \\
\!\!\! & = & \!\!\! \Lis_{(1,3)}(1,\theta^{2}) - \Lis_{4}(\theta^{2}) + \theta \Lis_{(1,3)}(1,1) - \theta \Lis_{4}(1) \\
\!\!\! & = & \!\!\! (-1)^{1-1} \Lis_{4}(\theta^{2}) + \theta \cdot (-1)^{1-1} \Lis_{4}(1) \\
& & + (-1)^{2-1} \Lis_{(1,3)}(1,\theta^{2}) + \theta \cdot (-1)^{2-1} \Lis_{(1,3)}(1,1), \\
\zetaAmot(1,3)
\!\!\! & = & \!\!\! (-1)^{1-1} \fLis_{4, \theta^{2}}(t) + t \cdot (-1)^{1-1} \fLis_{4, 1}(t) \\
& & + (-1)^{2-1} \fLis_{(1,3), (1,\theta^{2})}(t) + t \cdot (-1)^{2-1} \fLis_{(1,3), (1,1)}(t), \\
\end{eqnarray*}
and $(b_{1}(t), \fs_{1}, \bu_{1}) = (1, 4, \theta^{2})$, $(b_{2}(t), \fs_{2}, \bu_{2}) = (t, 4, 1)$, $(b_{3}(t), \fs_{3}, \bu_{3}) = (1, (1,3), (1,\theta^{2}))$, $(b_{4}(t), \fs_{4}, \bu_{4}) =  (t, (1,3), (1,1)) $.

For $\ell = 1$, we have $G_{1} = \bC^{\otimes 4}$,
and points
\[ \bv_{1} = (0, 0, 0, \theta^{2})^{\tr} \in \bC^{\otimes 4}(A), \]
\[ Z_{1} = \left( \left( \partial_{t}^{3} \fLis_{4, \theta^{2}} \right)|_{t=\theta}, \left( \partial_{t}^{2} \fLis_{4, \theta^{2}} \right)|_{t=\theta}, \left( \partial_{t}^{1} \fLis_{4, \theta^{2}} \right)|_{t=\theta}, \Lis_{4}(\theta^{2}))^{\tr} \in \Lie \bC^{\otimes 4}(\CC_{\infty} \right). \]
For $\ell = 2$, we have $G_{2} = \bC^{\otimes 4}$,
and points
\[ \bv_{2} = (0, 0, 0, 1)^{\tr} \in \bC^{\otimes 4}(A), \]
\[ Z_{2} = \left( \left( \partial_{t}^{3} \fLis_{4, 1} \right)|_{t=\theta}, \left( \partial_{t}^{2} \fLis_{4, 1} \right)|_{t=\theta}, \left( \partial_{t}^{1} \fLis_{4, 1} \right)|_{t=\theta}, \Lis_{4}(1) \right)^{\tr} \in \Lie \bC^{\otimes 4}(\CC_{\infty}). \]
We also have
\[ \rho_{2}(t)(\bv_{2}) = [t]_{4} \bv_{2} = (0, 0, 1, \theta)^{\tr} \in \bC^{\otimes 4}(A), \]
\[ \partial \rho_{2}(t)Z_{2} = \left( \left( \partial_{t}^{3} t \fLis_{4, 1} \right)|_{t=\theta}, \left( \partial_{t}^{2} t \fLis_{4, 1} \right)|_{t=\theta}, \left( \partial_{t}^{1} t \fLis_{4, 1} \right)|_{t=\theta}, \theta \Lis_{4}(1) \right)^{\tr} \in \Lie \bC^{\otimes 4}(\CC_{\infty}). \]
For $\ell = 3$, we have $G_{3} = \GG_{a}^{5}$ with the $t$-action
\[
\rho_{3}(t) = \left( \begin{array}{cccc|c} \theta & 1 & & & \\ & \theta & 1 & & \\ & & \theta & 1 & \\ \tau & & & \theta & - \theta^{2} \tau \\ \hline & & & & \theta + \tau \end{array} \right),
\]
and points
\[ \bv_{3} = (0, 0, 0, - \theta^{2}, 1)^{\tr} \in G_{3}(A), \]
\[ Z_{3}
= \begin{pmatrix}
- \left( \partial_{t}^{3} \fLis_{(1,3), (1,\theta^2)} \right)|_{t=\theta} \\
- \left( \partial_{t}^{2} \fLis_{(1,3), (1,\theta^2)} \right)|_{t=\theta} \\
- \left( \partial_{t}^{1} \fLis_{(1,3), (1,\theta^2)} \right)|_{t=\theta} \\
- \Lis_{(1,3)}(1,\theta^{2}) \\
\Lis_{1}(1) \end{pmatrix} \in \Lie G_{3}(\CC_{\infty}). \]
For $\ell = 4$, we have $G_{4} = \GG_{a}^{5}$ with the $t$-action
\[
\rho_{4}(t) = \left( \begin{array}{cccc|c} \theta & 1 & & & \\ & \theta & 1 & & \\ & & \theta & 1 & \\ \tau & & & \theta & - \tau \\ \hline & & & & \theta + \tau \end{array} \right),
\]
and 
\[ \bv_{4} = (0, 0, 0, - 1, 1)^{\tr} \in G_{4}(A), \]
\[ Z_{4}
= \begin{pmatrix}
- \left( \partial_{t}^{3} \fLis_{(1,3), (1,1)} \right)|_{t=\theta} \\
- \left( \partial_{t}^{2} \fLis_{(1,3), (1,1)} \right)|_{t=\theta} \\
- \left( \partial_{t}^{1} \fLis_{(1,3), (1,1)} \right)|_{t=\theta} \\
- \Lis_{(1,3)}(1,1) \\
\Lis_{1}(1)
\end{pmatrix} \in \Lie G_{4}(\CC_{\infty}). \]
We also have
\[ \rho_{4}(t) (\bv_{4}) = (0, 0, 1, \theta+1, \theta+1)^{\tr} \in G_{4}(A), \]
\[ \partial \rho_{4}(t) Z_{4}
= \begin{pmatrix}
- \left( \partial_{t}^{3} t \fLis_{(1,3), (1,1)} \right)|_{t=\theta} \\
- \left( \partial_{t}^{2} t \fLis_{(1,3), (1,1)} \right)|_{t=\theta} \\
- \left( \partial_{t}^{1} t \fLis_{(1,3), (1,1)} \right)|_{t=\theta} \\
- \theta \Lis_{(1,3)}(1,1) \\
\theta \Lis_{1}(1)
\end{pmatrix} \in \Lie G_{4}(\CC_{\infty}). \]
Therefore we have $G_{(1,3)} = \GG_{a}^{6}$ with the $t$-action
\[
\rho(t) = \left( \begin{array}{cccc|c|c}
\theta & 1 & & & & \\
& \theta & 1 & & & \\
& & \theta & 1 & & \\
\tau & & & \theta & - \theta^{2} \tau & -\tau \\ \hline
& & & & \theta + \tau & \\ \hline
& & & & & \theta + \tau
\end{array} \right),
\]
and 
\[ \bv_{(1,3)} = \pi(\bv_{1}, \rho_{2}(t)( \bv_{2}) , \bv_{3}, \rho_{4}(t) (\bv_{4})) = (0, 0, 0, 1, 1, \theta+1)^{\tr} \in G_{(1,3)}(A), \]
\[ Z_{(1,3)}
= \begin{pmatrix}
\left( \partial_{t}^{3} \zetaAmot(1,3) \right)|_{t=\theta} \\
\left( \partial_{t}^{2} \zetaAmot(1,3) \right)|_{t=\theta} \\
\left( \partial_{t}^{1} \zetaAmot(1,3) \right)|_{t=\theta} \\
(\theta^{2} + \theta) \zeta_{A}(1,3) \\
\Lis_{1}(1) \\
\theta \Lis_{1}(1)
\end{pmatrix} \in \Lie G_{(1,3)}(\CC_{\infty}). \]
\end{example}

\begin{corollary}\label{C: partial cZ}
Let notation and hypotheses be given in Theorem~\ref{T:MainThm}. Then for any polynomial $c(t)\in \FF_{q}[t]$, we have
\[  
\partial \rho(c(t)) Z_{\fs}=
\begin{pmatrix}
\left(\partial_{t}^{d_{1}-1}c(t) \zetaAmot(\fs)\right)|_{t=\theta} \\
\vdots\\
\left(\partial_{t}^{1} c(t)\zetaAmot(\fs)\right)|_{t=\theta}\\
c(t)\zetaAmot(\fs)|_{t=\theta}\\
\left(\partial \rho_{s+1}(c b_{s+1})(\bY_{s+1})\right)_{-}\\
\vdots\\
\left( \partial \rho_{T}(c b_{T}) (\bY_{T})\right)_{-}
\end{pmatrix}
=
\begin{pmatrix}
\left(\partial_{t}^{d_{1}-1} c(t)\zetaAmot(\fs)\right)|_{t=\theta} \\
\vdots\\
\left(\partial_{t}^{1}c(t)\zetaAmot(\fs)\right)|_{t=\theta}\\
c(\theta)\Gamma_{\fs}\zeta_{A}(\fs)\\
\left(\partial \rho_{s+1}(c b_{s+1})(\bY_{s+1})\right)_{-}\\
\vdots\\
\left( \partial \rho_{T}(c b_{T}) (\bY_{T})\right)_{-}
\end{pmatrix},\]where 
$\partial \rho_{\ell}(c b_{\ell})(\bY_{\ell}) $ is explicitly given in Corollary~\ref{C: partial bY} for $s+1\leq \ell \leq T$.
\end{corollary}
\begin{proof}
The arguments are entirely the same as above and so we omit them.
\end{proof}

\subsection{Monomials of MZV's} Given any two MZV's with weight $n_1$ and $n_2$, Thakur showed in \cite{T10} that the product of these two MZV's is an $\FF_p$-linear combination of certain MZV's of weight $n_{1}+n_{2}$, where $\FF_{p}$ is the prime field of $K$. It follows that for any indexes $\bk_1,\ldots,\bk_{m}$, there exist some distinct indexes $\fs_{1},\ldots,\fs_{n}$ of weight $w:=\wt(\bk_{1})+\cdots+\wt(\bk_{m})$ and coefficients $a_{1},\ldots,a_{n}\in \FF_{p}$ so that 
\begin{equation}\label{E: MonomialMZVs}
\zeta_{A}(\bk_{1})\cdots\zeta_{A}(\bk_{m})=a_{1}\zeta_{A}(\fs_{1})+\cdots+a_{n}\zeta_{A}(\fs_{n}). 
\end{equation}

For each $\fs_{i}$ above, we let $(G_{\fs_{i}},\rho_{\fs_{i}} )$ be the $t$-module over $A$ and $\bv_{\fs_{i}}\in G_{\fs_{i}}(A)$ be the special point defined in Sec.~\ref{Subsec: fiber coproduct}. Let $Z_{\fs_{i}}$ be given as in Theorem~\ref{T:MainThm}, and note that $\Exp_{G_{\fs_{i}}}(Z_{\fs_{i}})=\bv_{\fs_{i}}$. Recall that $G_{\fs_{i}}$ is the $t$-module associated to the dual $t$-motive $\cM_{\fs_{i}}$ containing $C^{\otimes w}$ as a sub-dual-$t$-motive. We put \[\cM:=\cM_{\fs_{1}}\sqcup_{C^{\otimes w}}\cdots \sqcup_{C^{\otimes w}}\cM_{\fs_{n}} ,\] which is the fiber coprodcut of $\left\{ \cM_{\fs_{i}}\right\}_{i=1}^{n}$ over $C^{\otimes w}$, and let $(G,\rho)$ be the $t$-module over $A$ associated to the dual $t$-motive $\cM$  in Sec.~\ref{Subsec: t-module G}. Using the arguments above, we can write $(G,\rho)$ explicitly as follows. 

We first note that for each $i$, $\rho_{\fs_{i}}(t)$ is a right upper triangular block matrix with $[t]_w$ located at upper left square in the sense that $\rho_{\fs_{i}}(t)$ has the shape of the form 
\[ 
\begin{pmatrix}
[t]_{w}& B_{\fs_{i}}\\
& \rho_{\fs_{i}}(t)'
\end{pmatrix}.
\]
We note that  if $\fs_{i} = (w)$, then $\rho_{\fs_{i}}(t) = [t]_{w}$, or in other words, the terms $B_{\fs_{i}}$ and $\rho_{\fs_{i}}(t)'$ are not present in this case.
Then $\rho(t)$ is given by
\[
\begin{pmatrix}
[t]_{w}&B_{\fs_{1}}&\cdots & B_{\fs_{n}}\\
& \rho_{\fs_{1}}(t)' & &\\
& & \ddots &\\
& & & \rho_{\fs_{n}}(t)'\\
\end{pmatrix}.
\] We further note that there is a natural morphism of $t$-modules 
\[\pi:=\left( \left(  \bz_{\fs_{1}},\dots,\bz_{\fs_{n}} \right) \mapsto \left( \sum_{i=1}^{n} \hat{\bz}_{\fs_{i}}^{\tr},{\bz_{\fs_{1}}^{\tr}}_{-},\ldots, {\bz_{\fs_{n}}^{\tr}}_{-}\right)^{\tr} \right): \bigoplus_{i=1}^{n}G_{\fs_{i}} \rightarrow G
,\] 
where $\hat{\bz}$ and ${\bz}_{-}$ are defined in Definition~\ref{Def: z-} by putting $d_{1}:=w$ there.

Using the methods of fiber coproduct~\cite[Lem.~3.3.2]{CM17b} and Theorem~\ref{T:MainThm}, we can relate the monomial $\zeta_{A}(\bk_{1})\cdots\zeta_{A}(\bk_{m})$ to the $w$-th coordinate of a certain logarithmic vector $Z$ and give explicit formulae for all the other coordinates. Before stating the result, we multiply both sides of (\ref{E: MonomialMZVs}) by $\Gamma_{\fs_{1}}\cdots\Gamma_{\fs_{n}}$ which gives
\begin{equation}\label{E: MonomialMZVs 2}
\Gamma_{\fs_{1}}\cdots\Gamma_{\fs_{n}}\zeta_{A}(\bk_{1})\cdots\zeta_{A}(\bk_{m})=c_{1}\Gamma_{\fs_{1}} \zeta_{A}(\fs_{1})+\cdots+c_{n} \Gamma_{\fs_{n}}\zeta_{A}(\fs_{n}),
\end{equation}
where for each $i$,
\begin{equation}\label{E: ci}
c_{i}:=a_{i}  \prod_{\substack{1\leq j\leq n \\ j\neq i}} \Gamma_{\fs_{j}}\in A
\end{equation}
and denote by $c_{i}(t):=c_{i}|_{\theta=t}\in \FF_{q}[t]$. Finally, we denote
\[Z:=  \partial\pi\left( \partial\rho_{\fs_{1}}(c_{1}(t)) Z_{\fs_{1}},\dots,\partial\rho_{\fs_{n}}(c_{n}(t)) Z_{\fs_{n}}\right)  \in \Lie G(\CC_{\infty}), \]
\[\bv:=\pi\left(\rho_{\fs_{1}}(c_{1}(t))( \bv_{\fs_{1}}),\dots,\rho_{\fs_{n}}(c_{n}(t))( \bv_{\fs_{n}}) \right)\in G(A),\]
so that $\Exp_{G}(Z)=\bv.$  Then, we have the following commutative diagram,

\[
\xymatrix{
{\displaystyle \bigoplus_{i=1}^{n}} G_{\fs_{i}} \ar@{->}[r]^{\ \ \pi}& G  \\
{\displaystyle \bigoplus_{i=1}^{n}} \Lie G_{\fs_{i}} \ar@{->}[r]^{ \ \ \ \partial \pi} \ar[u]^{\oplus_{i=1}^{n}\Exp_{G_{\fs_i}}} & \Lie G \ar[u]_{\Exp_{G}} .
}\]

\begin{theorem}\label{T: Monomial of MZV's}
 Let $\bk_{1},\ldots,\bk_{m}$ be $m$ indexes and put $w:=\wt(\bk_{1})+\cdots+\wt(\bk_{m})$.  Let $\left\{ \fs_{i},c_{i}\right\}_{i=1}^{n}$ be given in (\ref{E: MonomialMZVs 2}) and (\ref{E: ci}), and $Z_{\fs_{i}}$ be given in Theorem~\ref{T:MainThm} for each $i$. Let $G$ be the $t$-module, $Z\in \Lie G(\CC_{\infty})$ be the logarithmic vector and $\bv\in G(A)$ be the special point defined above. Then we have 
\begin{enumerate}
\item
$Z$ is given by
\[
\begin{pmatrix}
 {\displaystyle \sum_{i=1}^{n}}  \widehat{\partial \rho_{\fs_{i}}(c_{i}(t))Z_{\fs_{i}}} \\
\left(\partial \rho_{\fs_{1}}(c_{1}(t))Z_{\fs_{1}}\right)_{-}\\
\vdots\\
\left(\partial \rho_{\fs_{n}}(c_{n}(t))Z_{\fs_{n}}\right)_{-}
\end{pmatrix},
\]where $\partial \rho_{\fs_{i}}(c_{i}(t))Z_{\fs_{i}}$ is explicitly given in Corollary~\ref{C: partial cZ} for each $1\leq i \leq n$ and where we recall the notation from Definition~\ref{Def: z-}. 
\item The $w$-th coordinate of $Z$ is given by $ \Gamma_{\fs_{1}}\cdots\Gamma_{\fs_{n}}\zeta_{A}(\bk_{1})\cdots\zeta_{A}(\bk_{m})$.
\end{enumerate}
\end{theorem}
\begin{proof}
The first assertion follows from the definition of $\partial \pi$. To prove the second one, we first note that for each $i$, we have the following.
\begin{itemize}
\item $G_{\fs_{i}}$ comes from the dual $t$-motive $\cM_{\fs_{i}}$, and contains $\bC^{\otimes w}$ as sub-$t$-module.
\item  $\cM$ is the fiber coproduct of $\left\{ \cM_{\fs_{i}}\right\}_{i=1}^{n}$ over $C^{\otimes w}$ and $G$ is its corresponding $t$-module.
\item $\Exp_{G_{\fs_i}}(\partial \rho_{\fs_{i}}(c_{i}(t)) Z_{\fs_{i}})=\rho_{\fs_{i}}(c_{i}(t))\left(\bv_{\fs_{i}}\right)$.
\item The $w$-th coordinate of $\partial \rho_{\fs_{i}}(c_{i}(t))Z_{\fs_{i}}=c_{i}(\theta) \Gamma_{\fs_{i}} \zeta_{A}(\fs_{i})$.
\end{itemize}
So by \cite[Lem.~3.3.2]{CM17b} the $w$-th coordinate of $Z$ is given by $c_{1}\Gamma_{\fs_{1}} \zeta_{A}(\fs_{1})+\cdots+c_{n} \Gamma_{\fs_{n}}\zeta_{A}(\fs_{n})$, which is equal to $\Gamma_{\fs_{1}}\cdots\Gamma_{\fs_{n}}\zeta_{A}(\bk_{1})\cdots\zeta_{A}(\bk_{m})$ by (\ref{E: MonomialMZVs 2}).

\end{proof}

\begin{remark}
We mention that the coefficients $\left\{ a_{i}\right\}$ in (\ref{E: MonomialMZVs}) are shown to exist by Thakur \cite{T10}, and in general we do not know how to write them down explicitly. In the simplest case that $m=2$ and $\dep(\bk_1)=\dep(\bk_2)=1$, Chen has explicit formulae for the coefficients $\left\{ a_{i}\right\}$ in~\cite{Ch15}. Precisely, we have that for any positive integers $s_1$ and $s_2$ with $n:=s_1+s_2$, then
\begin{equation}\label{E:Chen}
\begin{aligned}\zeta_{A}(s_1) \zeta_{A}(s_2)&= \zeta_{A} (s_1,s_2) +  \zeta_{A} (s_2,s_1) + \zeta_{A} (s_1+s_2) \\
&+  \sum_{i+j =n  \atop   (q-1) | j}  \left[(-1)^{s_{1}-1}\binom{j-1}{s_{1}-1}+(-1)^{s_{2}-1}\binom{j-1}{s_{2}-1}  \right] \zeta_{A} (i, j). \ \ \ \
\end{aligned}
\end{equation}

\end{remark}

\end{document}